\documentclass[10pt,reqno]{article}

\usepackage{amssymb}
\usepackage{mathtools}
\usepackage{amsthm}
\usepackage{amsfonts}
\usepackage{amscd}
\usepackage{mathrsfs}

\usepackage{faktor}
\usepackage{xfrac}

\usepackage{graphicx}
\usepackage[export]{adjustbox}
\usepackage[dvipsnames]{xcolor}
\usepackage{colortbl}

\usepackage[margin=1in]{geometry}
\usepackage{setspace}
\usepackage{float}

\usepackage[inline]{enumitem}
\usepackage{subcaption}
\usepackage{appendix}
\usepackage{xparse}
\usepackage{tocloft}

\usepackage{hyperref}
\hypersetup{
    colorlinks=true,
    linkcolor=blue,
    citecolor=cyan,
    filecolor=magenta,
    urlcolor=cyan,
    pdfpagemode=FullScreen,
}
\usepackage{authblk}

\usepackage{nameref}
\usepackage[nocompress, space]{cite}

\usepackage{comment}
\usepackage{esint}
\usepackage{lineno}

\graphicspath{{./plots/}}

\numberwithin{equation}{section}

\theoremstyle{plain}
\newtheorem{thm}{Theorem}[section]
\newtheorem{theorem}[thm]{Theorem}
\newtheorem*{thm*}{Theorem}

\newtheorem{corollary}[thm]{Corollary}
\newtheorem*{cor*}{Corollary}
\newtheorem{lemma}[thm]{Lemma}
\newtheorem*{lemma*}{Lemma}

\newtheorem{proposition}[thm]{Proposition}
\newtheorem*{prop*}{Proposition}

\theoremstyle{definition}
\newtheorem{egsample}[thm]{Example}

\newtheorem{definition}[thm]{Definition}

\counterwithout{problem}{section}

\newtheorem{remark}[thm]{Remark}

\counterwithout{assumption}{section}

\newtheoremstyle{Customption}
  {\topsep}
  {\topsep}
  {\normalfont}
  {}
  {\bfseries}
  {:}
  {.5em}
  {}
\theoremstyle{Customption}

\newtheorem{innercustomAssumsA}{Assumptions}
\newenvironment{customAssumsA}[1]{
    \renewcommand\theinnercustomAssumsA{#1}
    \innercustomAssumsA
}{\endinnercustomAssumsA}

\newtheorem*{theorem*}{Theorem}

\newtheorem{mfg}{MFG System}
\counterwithout{mfg}{section}

\newenvironment{eg*}
{\begin{egsample*}}
{\leavevmode\unskip\penalty9999 \hbox{}\nobreak\hfill\quad\hbox{$\blacktriangle$}
\end{egsample*}}

\DeclareMathOperator*{\Div}{div}

\DeclareMathOperator{\avg}{avg}

\newcommand{\R}{\mathbb{R}}

\def \cX {\mathcal{X}}
\def \cY {\mathcal{Y}}

\newcommand{\cQ}{\mathcal{Q}}

\def \cI {\mathcal{I}}
\def \cH {\mathcal{H}}

\def\Chi{\scalebox{1.18}{\ensuremath{\chi}}}

\def \a {\alpha}
\def \b {\beta}

\def \th {\vartheta}

\def \dd {\delta}

\def\dx{{\,\,\rm d}x}

\def\dt{{\,\,\rm d}t}
\def\ds{{\,\,\rm d}s}

\begin{document}

\title{A Monotone Operator Approach to Separable Mean-Field Games with Mixed Boundary Conditions
}

\author[2,1]{AbdulRahman M. Alharbi}
\author[1]{Diogo Gomes}

\affil[1]{King Abdullah University of Science and Technology (KAUST)}
\affil[2]{The Islamic University of Al-Madinah}

\maketitle
\begin{abstract}
	We study a class of local, first-order, stationary mean-field games (MFGs) on bounded domains with nonstandard mixed boundary conditions: prescribed inflow on $\Gamma_N$ and a relaxed Signorini-type exit condition on $\Gamma_D$ (complementarity between exit flux and boundary value).
    For separable Hamiltonians, we overcome the lack of coercivity and the boundary complementarity constraints by introducing a monotone operator on a convex domain, augmented with an auxiliary nonnegative boundary variable $h$ encoding exit flux.
    To address a constant-shift degeneracy in the value function $u$ (the transport equation depends only on $Du$), we employ a quotient-space formulation that restores coercivity. Using the Browder--Minty theorem,
    we prove existence for a penalized operator $A_\epsilon$ on a convex domain and pass to the limit as $ \epsilon \to 0^+$.  
We obtain weak solutions $(m,u,h)$ solving the associated variational inequality, with $m \in L^{\beta+1}(\Omega)$, $u \in W^{1,\gamma}(\Omega)$, and $h$ in the dual trace space on $\Gamma_D$.
\end{abstract}

\section{Introduction} \label{Sec:Intro}

Mean Field Games (MFGs) provide a rigorous framework for analyzing strategic interactions in systems comprising a continuum of small players. 
Independently introduced by Huang, Malham\'e, and Caines \cite{huangLargePopulationStochastic2006,Caines2} and by Lasry and Lions \cite{ll1,ll2,lasryMeanFieldGames2007}, the theory characterizes equilibria via coupled systems of partial differential equations: a Hamilton--Jacobi (HJ) equation governing individual optimal control, and a transport equation governing the population distribution. These models have been applied in crowd dynamics \cite{achdouMeanFieldGames2019,toumi_tractable_2020,ghattassiNonseparableMeanField2025}, economics \cite{moll}, and finance \cite{CarmonaDelarueLacker2017,gomes2023priceagent}.

Variational techniques apply to first-order separable MFG systems \cite{Card1order,cardaliaguetMeanFieldGames2014}. Monotone operator methods extend this framework to low-regularity regimes: the monotonicity of the coupling corresponds to convexity of the associated functional, yielding existence, uniqueness, and stability  \cite{FG2,FGT1, FeGoTa21,FerreiraGomesTada2025,ferreiraWeakstrongUniquenessSolutions2025,ferreiraSolvingMeanFieldGames2025}. This structure admits monotone-operator discretizations
\cite{almullaTwoNumericalApproaches2017,GJS2,nurbekyanMonotoneInclusionMethods2024}. More recently, displacement-type or ``hidden'' monotonicity has been identified for non-separable Hamiltonians and master equations, extending the theory beyond classical separable settings \cite{gangboMeanFieldGames2022a,meszarosMeanFieldGames2024,bansilHiddenMonotonicityCanonical2024,bertucciMonotoneSolutionsMean2023}.
We establish an existence theory for local, stationary first-order separable MFGs subject to mixed inflow/Signorini boundary conditions. Our approach centers on an augmented monotone-operator formulation on a convex cone; when paired with a quotient-space reduction to restore coercivity in $u$, this framework provides a systematic operator-theoretic alternative to the variational methods developed in \cite{Alharbi2026MFG}.

\subsection{The Model Problem}
We analyze the operator associated with the following separable MFG system.
\begin{mfg}[\bfseries The Separable MFG]\label{mfg:2}
Let \(\Omega \subset \mathbb{R}^d\) be an open, bounded, and connected domain with a \(C^1\) boundary, denoted by \(\Gamma:= \partial \Omega\).
Throughout, $\nu$ denotes the outward unit normal on $\Gamma$, and $\cH^{d-1}$ denotes the $(d-1)$-dimensional Hausdorff measure.
We assume that \(\Gamma\) is partitioned into two relatively open, $C^1$ \((d-1)\)-dimensional manifolds: a Dirichlet boundary \(\Gamma_D\) and a Neumann boundary \(\Gamma_N\) such that
\[
\Gamma = \overline{\Gamma_D \cup \Gamma_N},\quad
\Gamma_D \cap \Gamma_N = \emptyset,\quad
\text{and }\cH^{d-1}(\partial\Gamma_D\cap\partial\Gamma_N)=0,
\]
and assume that $\Gamma_D$ and $\Gamma_N$ are both nonempty (hence of positive $\cH^{d-1}$-measure).
Additionally, let
\[
H:\Omega \times \R^d \to \R, \quad g: [0,\infty) \to \R, \quad \text{and }\enskip j:\Gamma_N \to [0,\infty)
\]
be given functions such that the map \(p \mapsto H(x,p)\) is convex, \( g \) is strictly increasing, and \( j\not\equiv 0 \). The MFG we study is given by
\begin{equation}\label{eq:MFG2sys}
	\begin{cases}
		H(x, Du) = g(m) & \text{in } \Omega, \\
		-\Div(m\,D_pH(x, Du)) = 0 & \text{in } \Omega,
	\end{cases}
\end{equation}
subject to the boundary conditions
\begin{equation}\label{eq:MFG2bc}
\begin{cases}
	 m\,D_pH(x, Du) \cdot \nu = j(x) & \text{on } \Gamma_N, \\
		m\,D_pH(x, Du) \cdot \nu \le 0, \quad u(x) \le 0 & \text{on } \Gamma_D, \\
		u\,m\,D_pH(x, Du) \cdot \nu = 0 & \text{on } \Gamma_D.
\end{cases}
\end{equation}
\end{mfg}

We call \eqref{eq:MFG2sys} \emph{separable} since the coupling enters additively: $\mathcal{H}(x,p,m)=H(x,p)-g(m)$. For the boundary conditions, the population flux is $J=-mD_pH(x,Du)$.
On $\Gamma_N$, the condition $mD_pH(x,Du)\cdot\nu=j$ prescribes an inflow of magnitude $j$ (since $J\cdot\nu=-j<0$).
On $\Gamma_D$, setting $h := J \cdot \nu = -m\,D_pH(x,Du) \cdot \nu$, the conditions take the Signorini form: $u \le 0$, $h \ge 0$, and $uh = 0$. This complementarity structure, reminiscent of the thin obstacle problem \cite{Signorini1959} and the boundary obstacle problem \cite{Athanasopoulos2008}, means the natural weak formulation is a variational inequality rather than a classical boundary-value problem.

The foundation for studying MFGs with such conditions was established in \cite{Alharbi2026MFG}, where a monotone MFG operator akin to that in \cite{FG2} was derived. However, the boundary complementarity forces the natural operator domain to be non-convex, which prevents the application of Browder--Minty-type existence theory. We overcome this by introducing an augmented operator with an auxiliary boundary flux variable, proving existence on a convex set, and recovering coercivity via a quotient-space reduction.
While \cite{Alharbi2026MFG} derives the operator structure and establishes existence via a variational approach, here we obtain existence (and the $\epsilon\to0$ limit) by convexifying the operator domain through the auxiliary boundary flux variable.

\subsection{Methodology and Operators}

We construct a MFG operator on a convex domain that naturally incorporates the boundary conditions. We achieve this by introducing an auxiliary variable $h \ge 0$ on $\Gamma_D$ encoding the exit flux $h = J \cdot \nu$.
This allows us to encode the Signorini complementarity through a variational inequality on a convex cone, while keeping the operator domain convex.

Let $\gamma > 1$ be a key growth exponent (see Section~\ref{sec:Nota&Assum}), and let $\gamma' = \gamma/(\gamma-1)$ be its Hölder conjugate. We consider triplets $(m,u,h)$ in a product of Lebesgue and Sobolev spaces. We define the limit operator $A_0$ and a penalized version $A_\epsilon$ (for $\epsilon > 0$) as follows:
\begin{equation}\label{eq:SepOpAIntro_0}
	A_0 \begin{bmatrix} m \\ u \\ h \end{bmatrix} := \begin{bmatrix}
		-H(x,Du) + g(m) \\
    -\Div \big(m\,D_pH(x,Du) \big) + \big(m\,D_pH(x,Du)\cdot \nu\,\Chi_{\Gamma} - j\,\Chi_{\Gamma_N} + h\,\Chi_{\Gamma_D}\big)\cH^{d-1}\\
		-u \,\,\Chi_{\Gamma_D}\cH^{d-1}
	\end{bmatrix},
\end{equation}
and
\begin{equation}\label{eq:NSepOpAIntro}
	A_\epsilon \begin{bmatrix} m \\ u \\ h \end{bmatrix} := \begin{bmatrix}
		-H(x,Du) + g(m) \\
		-\Div \big(m\,D_pH(x,Du) \big) + \big(m\,D_pH(x,Du)\cdot \nu\,\Chi_{\Gamma} - j\,\Chi_{\Gamma_N} + h\,\Chi_{\Gamma_D}\big)\cH^{d-1}\\
		\big(-u + \epsilon^{\gamma'}\,h^{\gamma'-1}\big)\Chi_{\Gamma_D}\cH^{d-1}
	\end{bmatrix},
\end{equation}
where $\Chi_S$ denotes the characteristic function of a set $S$.
These expressions are formal; the rigorous definitions via duality and trace-space pairings are given in Section~\ref{sec:MonoOpSepCase}.

In earlier work (e.g. \cite{FG2,FGT1, ferreiraSolvingMeanFieldGames2025}) the loss of coercivity in MFG operators was typically handled by adding a penalization term. In the present setting, the obstruction is linked to the mixed boundary complementarity on $\Gamma_D$ and the constant-shift invariance of the transport equation;
since 
the HJ and transport equations depend only on $Du$ and the boundary terms depend only on fluxes, the shift $u \mapsto u + c$ is invisible once $\int_{\Gamma_D} h = \int_{\Gamma_N} j$. So, coercivity must be restored on $W^{1,\gamma}/\mathbb{R}$.
Accordingly, we employ a boundary-driven regularization and a reduction strategy:
\begin{enumerate}
    \item We define the operator on a convex cone of nonnegative densities and fluxes.
    \item The boundary-only penalty (the $\epsilon$-term in \eqref{eq:NSepOpAIntro}) restores coercivity in $(m,h)$ for fixed $u$, allowing us to solve for $(m_u, h_u)$ explicitly.
    \item Substituting these back into the system produces a reduced monotone operator on the quotient space $W^{1,\gamma}(\Omega)/\mathbb{R}$, which removes the constant-shift degeneracy and restores coercivity in $u$.
\end{enumerate}
Our existence proof relies on applying the Browder--Minty theorem to this reduced operator, followed by passing to the limit as $\epsilon \to 0^+$.

\subsection{Main Results}
To state our results precisely, we introduce the functional setting.
The assumptions governing the exponents $\alpha,\beta,\gamma$ are stated in Assumption~\ref{assume:D.1}, while structural conditions on $H$ and $g$ are given in Assumptions~\ref{assumset:S}.

Let \( \cX:= L^{\beta+1}(\Omega)\times W^{1,\gamma}(\Omega) \times L^{\gamma'}(\Gamma_D) \), and let \( \cX^+ \subset \cX \) denote the convex cone of triplets with nonnegative density and boundary flux:
\[
\cX^+ := \left\{ (m,u,h) \in \cX : m \ge 0 \text{ a.e. in } \Omega \text{ and } h \ge 0 \text{ a.e. on } \Gamma_D \right\}.
\]
Our first main result establishes the existence of solutions for the penalized operator.

\begin{theorem}\label{thm:ExistenceOfSolutionsForSepA_ep_Intro}
Suppose that Assumptions~\ref{assumset:D} and \ref{assumset:S} hold, and let \( \epsilon > 0 \). Then, there exists a triplet \( (m_{\epsilon} ,u _{\epsilon},h_{\epsilon}) \in \cX^+ \) such that
\[
\left\langle A_{\epsilon} \begin{bmatrix} m_{\epsilon} \\ u_{\epsilon} \\ h_{\epsilon} \end{bmatrix}\, , \, \begin{bmatrix} \mu-m_{\epsilon} \\ v-u_{\epsilon} \\ k-h_{\epsilon} \end{bmatrix} \right \rangle\geq 0
\]
for all \( (\mu,v,k) \in \cX^+ \).
\end{theorem}

Our second result establishes the existence of a solution to the variational inequality associated with the limit operator \( A_0 \). In this limit, the boundary flux \( h \) relaxes to a distribution in the dual trace space \( W^{-1+\frac{1}{\gamma},\gamma'}(\Gamma_D) \).
Here we identify $k \in L^{\gamma'}(\Gamma_D)$ with its canonical embedding into $W^{-1+\frac{1}{\gamma},\gamma'}(\Gamma_D)$ when forming $k - h$, so that the duality pairing is well-defined.

\begin{theorem} \label{thm:ExistInSepMFG}
	Suppose that Assumptions~\ref{assumset:D} and \ref{assumset:S} hold. Then, there exists a triplet
	\[
	(m,u,h) \in L^{\beta+1}(\Omega)\times W^{1,\gamma}(\Omega) \times W^{-1+\frac{1}{\gamma},\gamma'}(\Gamma_D)
	\]
	such that \( m \ge 0 \), \( h \) is a nonnegative distribution, and
	\[
	\left\langle A_0 \begin{bmatrix} m \\ u \\ h \end{bmatrix}, \begin{bmatrix} \mu-m \\ v-u \\ k-h \end{bmatrix} \right\rangle \geq 0 \qquad \forall (\mu,v,k) \in \cX^+.
	\]
\end{theorem}
Theorem~\ref{thm:ExistInSepMFG} is obtained by solving the penalized variational inequalities for \(A_\epsilon\) using Theorem 
~\ref{thm:ExistenceOfSolutionsForSepA_ep_Intro} and passing to the limit as \(\epsilon\to 0^+\).
The remainder of this paper is organized as follows. In Section~\ref{sec:Nota&Assum}, we introduce our notation and assumptions, and review preliminary results including the Browder--Minty theorem and auxiliary estimates. In Section~\ref{sec:MonoOpSepCase}, we motivate the penalized MFG operator via a variational principle and rigorously define its functional framework. In Section~\ref{sec:KeyPropSepCase}, we establish the fundamental properties of the operator and detail its lack of global coercivity. In Section~\ref{sec:SepExiste}, we address this failure of coercivity via a quotient-space formulation and establish Theorem~\ref{thm:ExistenceOfSolutionsForSepA_ep_Intro}. Finally, in Section~\ref{sec:backToMFG}, we pass to the vanishing penalty limit to establish Theorem~\ref{thm:ExistInSepMFG}.

\paragraph{Acknowledgments:} The authors would like to thank their colleague Melih Ucer (KAUST) for valuable discussions that helped reduce the redundancy in our assumptions and improve the exposition. The first author also wishes to extend sincere thanks to Hicham Kouhkouh (Graz) for recommending a key reference that motivated the direction of this work.

\section{Notation, Assumptions, and Preliminary Results} \label{sec:Nota&Assum}

In this section, we establish the functional and technical framework necessary for the analysis of the Mean-Field Game operator. We begin by defining the notation and conventions used throughout the paper. Next, we specify the exponent regime and structural conditions on $H$ and $g$—requirements that ensure the operator $A_\epsilon$ is well-defined, monotone, and coercive enough to satisfy the hypotheses of the Browder–Minty theorem. These conditions, specifically the polynomial growth assumptions, are standard in the variational MFG literature (see, e.g., \cite{Card1order}). Finally, we establish several auxiliary bounds on the Hamiltonian and its gradient that are essential for the continuity and convergence arguments in the subsequent sections.

\subsection{Notation}
Before outlining the main assumptions, we discuss our notation and conventions throughout the paper.
\begin{enumerate}[label=\textbullet, leftmargin=*]
   \item \textbf{Differentiation:} Any differential operator without a subscript, such as \( Du \) or \( \Div F \), denotes differentiation with respect to the spatial variable \( x \). Differentiation with respect to any other variable is explicitly indicated (e.g., \( D_p H(x, p) \)).

   \item \textbf{Constants:} An inequality involving a constant \(C>0\) is understood to hold for some sufficiently large \(C\), independent of all free variables; dependencies on parameters are indicated by subscripts (e.g., $C_\delta$). The value of \(C\) may change from line to line.

   \item \textbf{Positive/Negative Parts:} For any quantity \( q \), we denote \( q_+ := \max\{q,0\} \) and \( q_- := \max\{-q,0\}\).

    \item \textbf{Conjugates:} We denote the H\"older conjugate of $q > 1$ by $q' := q/(q-1)$.

    \item \textbf{Extended Inverse:} 
    Throughout, $g^{-1}$ denotes the \emph{extended inverse} introduced in Definition~\ref{def:extendedInverse} (not the usual inverse on the range of $g$).

\item \textbf{Surface Measure:}
On the boundary $\Gamma$, we write $\ds$ for the surface measure $\mathrm d\cH^{d-1}$; thus
$\int_{\Gamma} f\,\ds := \int_{\Gamma} f\,\mathrm d\cH^{d-1}$.
For measurable \(E\subseteq \Gamma\), we also write \(|E|:=\int_E 1\,\ds=\cH^{d-1}(E)\).

\item \textbf{Traces:} 
    For $u \in W^{1,\gamma}(\Omega)$, we write $Tu$ or $u|_{\Gamma}$ interchangeably for its trace on $\Gamma$, and $u|_{\Gamma_D}$ for the restriction to $\Gamma_D$.

\item \textbf{Boundary integrals and normal traces:}
Unless explicitly stated otherwise, boundary integrals $\int_{\Gamma_*} u\,\ds$ for $u\in W^{1,\gamma}(\Omega)$ are understood using the trace $Tu$.
When we write $F\cdot \nu$ for a vector field $F$, we mean the \emph{normal trace} in $W^{-1+\frac{1}{\gamma},\gamma'}(\Gamma)$ (when it is well-defined, e.g.\ if $F\in L^{\gamma'}(\Omega;\R^d)$ and $\Div F\in L^{\gamma'}(\Omega)$), paired with traces via $\langle\cdot,\cdot\rangle_\Gamma$.

    \item \textbf{Duality:} 
    The pairing between a Banach space $\cX$ and its dual $\cX'$ is denoted $\langle \cdot, \cdot \rangle$, which reduces to $\int_{\Omega} fg \dx$ or $\int_{\Gamma} fg \ds$ for Lebesgue and standard Sobolev spaces. For boundary terms in $W^{-1+1/\gamma,\gamma'}(\Gamma)$, we write $\langle \cdot, \cdot \rangle_{\Gamma}$ for the trace-space duality pairing.

    \item \textbf{Convergence:} We write $u_n\rightharpoonup u$ (resp. $u_n\to u$) for weak (resp. strong) convergence.

\end{enumerate}

\subsection{Assumptions} \label{subsec:Assumptions}

The assumptions detailed here address the problem data and the separable structure of the Hamiltonian, ensuring the well-posedness of MFG System~\ref{mfg:2} and the underlying control problem.

\subsubsection{The problem data}

We begin with our assumptions on the basic parameters and data. We denote these assumptions by {\textbf{D}} (for ``data'').
\begin{customAssumsA}{\textbf{D}} \label{assumset:D} \,
  \begin{enumerate}[label=\textbf{D.\arabic*}, leftmargin=*]

  \item \label{assume:D.1} Our assumptions are formulated in terms of two key exponents:
  \begin{equation}\label{eq:parameters1}
    \alpha > 1 \, \text{ and } \, \beta > 0.
  \end{equation}
  To simplify our presentation, we additionally introduce an auxiliary exponent:
  \begin{equation}\label{eq:parameters2}
  \gamma := \dfrac{\beta + 1}{\beta} \alpha. 
  \end{equation}
Note that  $\gamma>\alpha>1$ since $\beta>0$ and $\alpha>1$. This choice balances the growth of the coupling $g(m)$ and the Hamiltonian $H(Du)$ in the energy estimates.

  \item \label{assume:D.2} The incoming flow \( j \in L^{\gamma'}(\Gamma_N) \) is nonnegative and not identically zero; that is,
  \[
  j \geq 0 \quad \text{ and } \quad j\not \equiv 0.
  \]
  \end{enumerate}
  \end{customAssumsA}

We recall that $\Gamma_N\subset\partial\Omega$ denotes the (Neumann/inflow) boundary portion where the incoming flux is prescribed. The integrability condition in the preceding assumption is required for the weak formulation of the transport equation, ensuring the boundary source term is well-defined.

\subsubsection{The Hamiltonian}

The following assumptions, denoted by \(\mathbf{S}\) (for “separable”), govern the regularity and growth of the Hamiltonian $H$ and the coupling term $g$. These conditions are essential for the well-posedness of MFG System~\ref{mfg:2} and distinguish our framework from that of a companion paper on nonseparable MFGs \cite{alharbiNonSepMonoOperator2025}.

\begin{customAssumsA}{\textbf{S}} \label{assumset:S}
\,
\begin{enumerate}[label=\textbf{S.\arabic*}, start=0]
  \item \label{assume:S.0}
  \begin{enumerate}[label= \textbf{\alph*}., ref=\textbf{\alph*}]
    \item\label{assume:S.0.a} The coupling term \( g:[0, \infty) \to \mathbb{R} \) is continuous.

    \item\label{assume:S.0.b} The Hamiltonian \(H:\bar{\Omega} \times \mathbb{R}^d\to \R\) is measurable in \(x\) and \(H(x,\cdot)\in C^1(\mathbb{R}^d)\) for almost every \(x\in\Omega\).
  \end{enumerate}

  \item[]
  This regularity guarantees that the terms in the Hamilton-Jacobi and transport equations are well-defined.

  \item \label{assume:S.1} The coupling term $g$ is strictly increasing.

\item[] 
This assumption represents players' aversion to congestion
and ensures that \(g^{-1}\) exists (see Definition~\ref{def:extendedInverse}). 

  \item \label{assume:S.2} The function \( H \) is convex in \( p \).

  \item[] This convexity is used twice: to (i) obtain monotonicity of the $H$-terms and (ii) control $m D_p H$ via growth/coercivity estimates. 

  \item \label{assume:S.3} There exists a constant $C>1$ such that, for all $m \in [0,\infty)$
  \[
    		C^{-1}m^{\beta} -C \leq g(m) \leq C m^{\beta} + C.
  \]

  \item[]
  The lower bound establishes the coercivity of the coupling term, while the upper bound controls its growth rate. Both properties are essential for the continuity and coercivity arguments used in the analysis of the MFG operator.

  \item \label{assume:S.4} There is a constant $C>1$ such that
    \begin{enumerate}[label= \textbf{\alph*}.]
    \item \quad \( |D_pH(x,p)| \leq C|p|^{\a-1} + C \) \enskip for a.e.\ \(x\in\Omega\) and all \(p\in\mathbb{R}^d\).

    \item \quad \( H(x,0) \leq C \) \enskip for a.e. $x \in \Omega$.
    \end{enumerate}

  \item[]
  This assumption provides an upper bound on the growth of the Hamiltonian and its derivative (see Remark~\ref{rmk:HamiltonImplicitBound}). This control is essential for two reasons: it guarantees that the underlying optimal control problem is well-defined, and it is a key ingredient in the proof of continuity for the MFG operator.
  \end{enumerate}

  \begin{remark} \label{rmk:HamiltonImplicitBound}
    Due to the fundamental theorem of calculus, for a.e. $x\in \Omega$, Assumption~\ref{assume:S.4} gives the following bound on the Hamiltonian
    \[
    H(x,p) \leq H(x,0) + \int_{0}^{1} |D_pH(x,tp)| |p| \dt \leq C|p|^{\alpha} + C.
    \]
  \end{remark}

The following coercivity assumption is stated in three forms; Proposition~\ref{prop:equivAssumS.5} shows they are equivalent under Assumptions~\ref{assume:S.2} and \ref{assume:S.4}.

  \begin{enumerate}[label=\textbf{S.\arabic*}, start=5]
  \item \label{assume:S.5} There exists a constant $C>1$ such that, for a.e. $x\in\Omega$ and all $p\in \mathbb{R}^d$, one of
  the following conditions holds.

  \begin{enumerate}[label= \textbf{\alph*}., ref=\textbf{\alph*}]
  \item \label{assume:S.5.a} \quad \( H(x,p) \geq C^{-1} |p|^\alpha-C \).

  \item \label{assume:S.5.b} \quad \( D_pH(x,p) \cdot p \geq C^{-1} |p|^\alpha-{C} \) \quad and \enskip\, \( H(x,p) \geq -C \).

  \item \label{assume:S.5.c} \quad \( - H(x,p) + D_pH(x,p) \cdot p \geq C^{-1}|p|^{\a}-C \) \quad and \enskip \, \( H(x,p) \geq -C \).

  \end{enumerate}

\item[]
Forms  \ref{assume:S.5.a} and \ref{assume:S.5.b}  establish growth bounds for the operator analysis. 
  \end{enumerate}
\end{customAssumsA}

\begin{proposition}
	\label{prop:equivAssumS.5}
	Under  Assumptions~\ref{assume:S.2}  and \ref{assume:S.4}, the three forms of \ref{assume:S.5} are equivalent.
\end{proposition}
\begin{proof}
	We show that $\text{\ref{assume:S.5.a}} \Rightarrow \text{\ref{assume:S.5.b}} \Rightarrow \text{\ref{assume:S.5.c}} \Rightarrow \text{\ref{assume:S.5.a}}$.

    \textbf{(a $\Rightarrow$ b)}. By the convexity of $H$ in $p$ (\ref{assume:S.2}), we have that \[H(x,0) \ge H(x,p) + D_pH(x,p)\cdot(0-p).\] Rearranging gives \[D_pH(x,p)\cdot p \ge H(x,p) - H(x,0).\] Using \ref{assume:S.5.a} for the lower bound on $H(x,p)$ and \ref{assume:S.4}.b for the upper bound on $H(x,0)$, we get \[D_pH(x,p)\cdot p \ge (C^{-1}|p|^\alpha - C) - C = C^{-1}|p|^\alpha - 2C.\] After adjusting the constant $C$, this gives the first inequality in \ref{assume:S.5.b}. The second, $H(x,p) \ge -C$, is an immediate consequence of \ref{assume:S.5.a}.

	\textbf{(b $\Rightarrow$ c)}. 
   Let $\sigma \in(0,1)$. By convexity (\ref{assume:S.2}), \[H(x,p) - H(x,\sigma p) \le (1-\sigma) D_pH(x,p)\cdot p.\]
     We rearrange this to get \[-H(x,p) + D_pH(x,p)\cdot p \ge \sigma D_pH(x,p)\cdot p - H(x,\sigma p).\] 
     To the right-hand side of this inequality, 
we apply the lower bound from \ref{assume:S.5.b} to the first term and the upper bound from Remark~\ref{rmk:HamiltonImplicitBound} (derived from \ref{assume:S.4}) to the second:
	\[
	-H(x,p) + D_pH(x,p)\cdot p \ge \sigma (C^{-1}|p|^\alpha - C) - (C|\sigma p|^\alpha + C) = (\sigma C^{-1} - C\sigma^\alpha)|p|^\alpha - C(\sigma+1).
	\]
	We can choose $\sigma > 0$ small enough such that the coefficient $(\sigma C^{-1} - C\sigma^\alpha)$ is positive. Adjusting the constants then yields the required inequality in \ref{assume:S.5.c}. The condition $H(x,p) \ge -C$ is given directly by \ref{assume:S.5.b}. 
   
	\textbf{(c $\Rightarrow$ a)}. We write $H(x,p) = H(x,p/2) + \int_{1/2}^1 D_pH(x,tp)\cdot p \dt$. Rearranging the integrand and applying \ref{assume:S.5.c} gives
	\begin{align*}
		H(x,p) &= H(x,p/2) + \int_{1/2}^1 \frac{1}{t} \left( D_pH(x,tp)\cdot(tp) - H(x,tp) \right)\dt + \int_{1/2}^1 \frac{1}{t} H(x,tp) \dt \\
		&\ge -C + \int_{1/2}^1 \frac{1}{t}\big(C^{-1}|tp|^\alpha - C\big)\dt + \int_{1/2}^1 \frac{1}{t}(-C)\dt \\
		&\ge -C - 2C\log(2) + C^{-1}|p|^\alpha \int_{1/2}^1 t^{\alpha-1}\dt = C^{-1} \frac{1-2^{-\alpha}}{\alpha}|p|^\alpha - C'.
	\end{align*}
	Adjusting constants gives \ref{assume:S.5.a}.
\end{proof}

The convexity of $H$ in $p$ (Assumption \ref{assume:S.2}) and its superlinear growth (Assumption \ref{assume:S.5}) ensure that the Legendre--Fenchel transform is well-defined. This allows us to introduce the Lagrangian $L(x,v)$ as the convex conjugate of $H(x,\cdot)$:$$L(x, v) = \sup_{p \in \mathbb{R}^d} \left[- p \cdot v - H(x, p) \right].$$Under these conditions, $L(x,v)$ is a proper, lower semi-continuous, and convex function that is also coercive in $v$. Furthermore, the involution property of the conjugate holds, dually recovering the Hamiltonian:$$H(x, p) = \sup_{v \in \mathbb{R}^d} \left[- v \cdot p - L(x, v) \right].$$
Form \ref{assume:S.5.c} of Assumption \ref{assume:S.5} corresponds to the coercivity of the Lagrangian $L(x,v^*) = -H + p \cdot D_pH$.

\paragraph{Example.} A prototype Hamiltonian for which the assumptions hold is
\[
H(x,p)= (|p|^2+1)^{\alpha/2} +1,
\]
with the following coupling term
\[
g(m)= (m^2+1)^{\beta/2}.
\]
Here $g(0) = 1$, so $g^{-1}(z) = 0$ for $z < 1$.
For $\alpha$ and $\beta$ as in Assumption \ref{assume:D.1},
these functions satisfy the growth, convexity, and coercivity conditions detailed in Assumptions~\ref{assumset:S}.

\subsection{Monotone Operators and the Browder--Minty Theorem} \label{sec:MonoMinty}

Following \cite{KiSt00}, 
we recall the Browder--Minty theorem for variational inequalities. Let \( \cX \) be a reflexive Banach space and let \( \cX' \) be its topological dual, with the duality pairing denoted by \( \langle\cdot,\cdot\rangle \). We begin by introducing the primary notions needed to apply the theorem: monotonicity, coercivity, and continuity on finite-dimensional subspaces.

\begin{definition}\label{def:monotonicity}
Let \( \mathcal{K} \subseteq \cX \) be a nonempty, closed, and convex set. A mapping \( A : \mathcal{K} \to \cX' \) is \emph{monotone} if
\[
\langle Au - Av,\; u - v \rangle \;\ge\; 0 \quad
\text{for all }u,v \in \mathcal{K}.
\]
It is \emph{strictly monotone} if the above inequality is strict for \( u \neq v \); that is,
\[
\langle Au - Av,\; u - v \rangle \;=\; 0  \;\; \implies \;\; u = v.
\]
\end{definition}

\begin{definition}\label{def:continuity}
Let \( \mathcal{K} \subseteq \cX \)  be a nonempty, closed, and convex set. A mapping \( A : \mathcal{K} \to \cX' \) is \emph{continuous on finite-dimensional subspaces} if, for every finite-dimensional subspace $\mathcal{M} \subset \cX$, the restriction
\[
A\bigl|_{\mathcal{K} \cap \mathcal{M}} : \mathcal{K} \cap \mathcal{M} \;\to\; \cX'
\]
is weakly continuous; that is, whenever $u_n \to u$ in $\mathcal{K}\cap\mathcal{M}$ (equivalently in $\mathcal{M}$),
we have $Au_n \rightharpoonup Au$ in $\cX'$ (with respect to the weak topology $\sigma(\cX',\cX)$).
\end{definition}

\begin{remark}
A mapping $A: \mathcal{K} \to \cX'$ is  \emph{hemicontinuous} if for all $u, v \in \mathcal{K}$ and $\varphi \in \cX$, the function
\begin{equation*}
    f(\theta) = \langle A((1 - \theta)u +\theta v)  ,  \varphi \rangle
\end{equation*}
is continuous for $\theta \in [0,1]$.
For monotone mappings, hemicontinuity is equivalent to continuity on finite-dimensional subspaces in the sense of Definition~\ref{def:continuity}; see, e.g., \cite{MR163198}.
\end{remark}

We recall the Browder--Minty theorem.

\begin{lemma}[\!\! \cite{KiSt00}, Theorem 1.7] \label{lem:BoundedMintyExistence}
Let $\mathcal{K} \subset \cX$ be a nonempty, closed, bounded, and convex subset, and let \( A: \mathcal{K} \to \cX' \) 
be monotone and continuous on finite-dimensional subspaces. 
Then, there exists $u \in \mathcal{K}$ such that
\begin{equation*}
\langle Au, v - u \rangle \geq 0 \quad \text{for all } v \in \mathcal{K}.
\end{equation*}
\end{lemma}

To extend the Browder--Minty theorem to an unbounded domain $\mathcal{K}$, the theorem requires an additional coercivity assumption. We define this notion as follows.
\begin{definition} \label{def:coercivity}
Let $\mathcal{K}\subseteq \cX$ be nonempty, closed, and convex. A mapping $A: \mathcal{K} \to \cX'$ is \emph{coercive} if there is $\varphi \in \mathcal{K}$ for which
\[
\frac{\langle Au - A\varphi, u - \varphi \rangle}{\|u - \varphi \|_{\cX}} \to +\infty \quad \text{as } \|u\|_{\cX} \to +\infty,  u\in\mathcal{K}.
\]
\end{definition}
This leads to the following extension of the Browder--Minty theorem.
\begin{lemma}[\!\! \cite{KiSt00}, Corollary 1.8]\label{lem:UnboundedMintyExistence}
Let $\mathcal{K} \subset \cX$ be a nonempty, closed, and convex set. Suppose \( A: \mathcal{K} \to \cX' \) is coercive, monotone, and continuous on finite-dimensional subspaces. Then, there exists $u \in \mathcal{K}$ such that
\[
\langle Au, v - u \rangle \geq 0 \quad \text{for all } v \in \mathcal{K}.
\]
\end{lemma}

If $\mathcal K=\mathcal X$ is a linear space, then $\langle Au,v-u\rangle\ge0$ for all $v\in\mathcal X$ is equivalent to $Au=0$ in $\mathcal X'$.

\subsection{Auxiliary Bounds on the Hamiltonian} \label{sec:auxOnSepH}

We establish estimates for $H$ and $g$ required to prove the hemicontinuity of the operator.

\subsubsection{Basic auxiliary bounds}

The estimates in this subsection utilize the \emph{extended inverse} of \( g \), a mapping that is central to the analysis of MFG~System~\ref{mfg:2}.
\begin{definition} \label{def:extendedInverse}
Given a non-decreasing function \( g:[0,\infty) \to \R \), an extended inverse \( g^{-1}:\R\to [0,\infty) \) is a mapping satisfying
    \begin{enumerate}[label=\roman*.]
        \item \( g(g^{-1}(z))=z \) for all \( z \geq g(0) \), and
        \item \( g^{-1}(z) = 0 \) for all \( z < g(0) \).
    \end{enumerate}
\end{definition}
Under Assumption~\ref{assume:S.1}, any extended inverse (when it exists) is unique and coincides with the usual inverse of $g$ on its range $g([0,\infty))$.
The following result shows that, under our standing assumptions on $g$, one has $g([0,\infty))=[g(0),\infty)$, so an extended inverse exists.
\begin{proposition} \label{prop:wellDefinedInverse}
Suppose \( g:[0,\infty) \to \R \)
satisfies Assumptions~\ref{assume:D.1}, \ref{assume:S.0}~\ref{assume:S.0.a}, \ref{assume:S.1},
        and \ref{assume:S.3}. Then,
    the extended inverse $g^{-1}$ is well-defined and unique.
\end{proposition}
\begin{proof}
For $z < g(0)$, the definition prescribes $g^{-1}(z)=0$. For $z \ge g(0)$, strict monotonicity of $g$ (Assumption~\ref{assume:S.1}) implies injectivity, hence there is at most one $m\ge 0$ such that $g(m)=z$.
Moreover, because $\beta>0$ by Assumption \ref{assume:D.1}, Assumption~\ref{assume:S.3} implies \(g(m)\to+\infty\) as \(m\to+\infty\); combined with continuity of $g$ (Assumption~\ref{assume:S.0.a}), this yields \(g([0,\infty))=[g(0),\infty)\), so a pre-image exists for every \(z\ge g(0)\).
\end{proof}
\begin{remark}
In Proposition~\ref{prop:wellDefinedInverse}, only the condition \(\beta>0\) from Assumption~\ref{assume:D.1} is used (the exponent \(\alpha\) plays no role).
\end{remark}

We now establish basic bounds for \( g^{-1} \).
\begin{lemma}\label{lem:boundOnginverse1}
Suppose \( g:[0,\infty) \to \R \)
satisfies Assumptions~\ref{assume:D.1}, \ref{assume:S.0}~\ref{assume:S.0.a}, \ref{assume:S.1} and \ref{assume:S.3}. Let \( g^{-1} \) be the extended inverse of \( g \). Then, for all \( z \in \R \), we have 
    \[
       \frac{(z-g(0))_+^{1/\beta}}{C}-C  \leq g^{-1}(z) \leq C(z-g(0))_+^{1/\beta}+C,
    \]
    where \( C > 1 \) is a constant independent of \( z \).
\end{lemma}
\begin{proof}
For \( z \leq g(0) \), we have \( g^{-1}(z) = 0 \) and \( (z-g(0))_+ = 0 \), so the inequalities hold. For \( z > g(0) \), let \( m = g^{-1}(z) \), which implies \( g(m) = z \). Assumption~\ref{assume:S.3} implies
\[
C^{-1} m^{\beta} - C \leq g(m) \leq C m^{\beta} + C.
\]
Since $g(0)$ is finite, we may absorb it into the generic constant. More precisely, from Assumption \ref{assume:S.3} and $z=g(m)$ we have
\[
C^{-1} m^{\beta} - C \le z \le C m^{\beta} + C,
\]
hence
\[
C^{-1} m^{\beta} - C' \le z - g(0) \le C m^{\beta} + C'
\]
for a constant $C'>0$ depending only on the constants in Assumption \ref{assume:S.3} and on $g(0)$.
Isolating $m$ and using that $z-g(0)=(z-g(0))_+$ for $z>g(0)$ yields the stated bounds (after renaming constants).
\end{proof}

\begin{corollary}\label{cor:boundOng^-1(H)1}
Suppose \( g:[0,\infty) \to \R \)
satisfies Assumptions~\ref{assume:D.1}, \ref{assume:S.0}, 
\ref{assume:S.1}, and \ref{assume:S.3}. Furthermore, assume that Assumptions~\ref{assume:S.4} and \ref{assume:S.5} hold.
Let \( g^{-1} \) be the extended inverse of \( g \).
Then, 
there exists \( C > 1 \) such that
for a.e.\ \(x\in\Omega\) and all \(p\in\R^d\), 
    \[
            \frac{|p|^{\alpha/\beta}}{C}-C  \leq g^{-1}(H(x,p)) \leq C|p|^{\alpha/\beta}+C.
    \]
\end{corollary}
\begin{proof}
The result follows from Lemma~\ref{lem:boundOnginverse1} with \( z = H(x,p) \), Remark~\ref{rmk:HamiltonImplicitBound}, and Assumption~\ref{assume:S.5}. Note that while the lower bound may be negative for small values of $p$, we have $g^{-1} \ge 0$. 
\end{proof}

\subsubsection{Convex combination bounds}

In this subsection, we establish pointwise growth estimates for the Hamiltonian and flux terms along convex combinations. These bounds serve as the integrable majorants required to apply the dominated convergence theorem when proving the hemicontinuity of the MFG operator (see Section~\ref{sec:KeyPropSepCase}). We begin with the estimate for \( H \).
\begin{lemma}\label{lem:convProdSC1}
    Suppose that Assumptions~\ref{assume:D.1} and \ref{assume:S.3}--\ref{assume:S.5} hold. 
    Let \( p_1, p_0 \in \R^d \) and \( m_1,m_0 \in [0,\infty) \) be arbitrary, and define
    \[
    p_\th = (1-\th)p_0 + \th p_1 \quad  \text{ and } \quad  m_\th = (1-\th)m_0 + \th m_1.
    \]
Then, for a.e.\ $x\in \Omega$ and for every \( \th \in [0,1] \), we have
    \[
        |H(x, p_\th) - g(m_\th)| \leq C \left( \th^{\min(\b, 1)}  \left( |p_1|^{\alpha} + m_1^{\beta}\right) + (1-\th)^{\min(\b, 1)}  \left( |p_0|^{\alpha} + m_0^{\beta} \right) +1 \right)
    \]
    where \( C > 1 \) is a constant independent of \( p_0, p_1, m_0, m_1, \) and \( \th \).
\end{lemma}
\begin{proof}
    From Assumptions~\ref{assume:S.3}--\ref{assume:S.5}, we have the upper bound
    \begin{equation}\label{eq:UBonSepH1-1}
    |H(x, p_\th) - g(m_\th)| \leq C (|p_\th|^{\a }+ m_\th^{\b }+1).
    \end{equation}
    The convexity of the mapping \( p \mapsto |p|^\alpha \) implies
    \begin{equation}\label{eq:UBonSepH1-3}
    |p_\th|^{\a } \leq (1-\th) |p_0|^{\a }+\th |p_1|^{\a }.
    \end{equation}
Since $\min(\b,1)\le 1$ and $0\le \th \le 1$, we have $\th \le \th^{\min(\b,1)}$ and $1-\th \le (1-\th)^{\min(\b,1)}$. Hence,
\[
(1-\th)|p_0|^{\a}+\th|p_1|^{\a}
\le (1-\th)^{\min(\b,1)}|p_0|^{\a}+\th^{\min(\b,1)}|p_1|^{\a}.
\]
    To bound \( m_\th^\b \), we distinguish two cases:
    \begin{enumerate}[label=\roman*.]
        \item If \( 0 < \beta \leq 1 \), the function $t \mapsto t^\beta$ is subadditive on $[0,\infty)$, which implies
        \[
        m_\th^\beta = ((1-\th)m_0 + \th m_1)^\beta \le (1-\th)^\beta m_0^\beta + \th^\beta m_1^\beta.
        \]
        \item If \( \b > 1 \), the convexity of the mapping \( m \mapsto m^{\beta} \) yields
        \[
        m_\th^\b \leq (1-\th) m_0^\beta+\th m_1^\b.
        \]
    \end{enumerate}
    In both cases, we have
    \begin{equation}\label{eq:UBonSepH1-2}
        m_\th^\b \leq (1-\th)^{\min(\b, 1)}\, m_0^\beta+\th^{\min(\b, 1)}\, {m_1^\b}.
    \end{equation}
    Combining \eqref{eq:UBonSepH1-1}, \eqref{eq:UBonSepH1-3}, and \eqref{eq:UBonSepH1-2} completes the proof.
\end{proof}

Our next lemma provides an upper bound on \( m D_pH(x,p) \), which governs the flow of agents in the system. This estimate is essential for ensuring the continuity of the monotone operator defined in Section~\ref{sec:MonoOpSepCase}.

\begin{lemma}\label{lem:convProdSC2}

Suppose that Assumptions~\ref{assume:D.1} and \ref{assume:S.4} hold. Let $p_1, p_0 \in \R^d$ and $m_1,m_0 \in [0,\infty)$ be arbitrary, and define $p_\th$ and $m_\th$ as in Lemma~\ref{lem:convProdSC1}.
Then, for a.e.\ \(x\in\Omega\) and every \( \th \in [0,1] \), we have
    \[
        |m_\th D_pH(x, p_\th)| \leq C \left( \th^{\min(\gamma-1, 1)} \left( |p_1|^{\gamma-1}+m_1^{(\beta+1)/\gamma'} \right)
        + (1-\th)^{\min(\gamma-1, 1)} \left( |p_0|^{\gamma-1}+m_0^{(\beta+1)/\gamma'} \right)  +1 \right),
    \]
    where \( C > 1 \) is a constant independent of \( p_0, p_1, m_0, m_1, \) and \( \th \).
\end{lemma}
\begin{proof}
    From Assumption~\ref{assume:S.4}, we have the upper bound
    \begin{equation} \label{eq:convProdSC2-1}
        |m_\th D_pH(x, p_\th)| \leq C\left( m_\th|p_\th|^{\a-1} + m_\th \right).
    \end{equation}
    To bound \( m_\th |p_\th|^{\a-1} \), we apply Young's inequality with exponents
    \[
        \lambda = \frac{\beta+1}{\gamma'} = \frac{\beta}{\a'}+1 > 1 \quad \text{ and } \quad \lambda'= \frac{\gamma-1}{\a-1}.
    \]
    A direct calculation using $\gamma = \frac{\beta+1}{\beta}\alpha$ verifies that $1/\lambda + 1/\lambda' = 1$. Applying Young's inequality yields
    \[
        m_\th |p_\th|^{\a-1} \leq \frac{m_\th^{\lambda}}{\lambda} + \frac{|p_\th|^{(\a-1)\lambda'}}{\lambda'} = \frac{m_\th^{(\beta+1)/\gamma'}}{\lambda} + \frac{|p_\th|^{\gamma-1}}{\lambda'}.
    \]
    Similarly, applying Young's inequality to the linear term gives $m_\th \leq \frac{m_\th^{\lambda}}{\lambda} + \frac{1}{\lambda'}$. Substituting these into \eqref{eq:convProdSC2-1}, we obtain
    \[
        |m_\th D_pH(x, p_\th)| \leq C\left( m_\th^{\lambda} +  |p_\th|^{\gamma-1}+1\right).
    \]
Finally, since $\lambda>1$ the map $t\mapsto t^\lambda$ is convex, hence
\[
m_\th^\lambda \le (1-\th)m_0^\lambda + \th m_1^\lambda \le (1-\th)^{\min(\gamma-1,1)} m_0^\lambda + \th^{\min(\gamma-1,1)} m_1^\lambda,
\]
because \(0\le \th \le 1\) implies \(\th \le \th^{\min(\gamma-1,1)}\) and \(1-\th \le (1-\th)^{\min(\gamma-1,1)}\).
To bound the remaining term $|p_\th|^{\gamma-1}$, we use $|p_\th|\le (1-\th)|p_0|+\th|p_1|$.
If $\gamma\ge 2$, then $t\mapsto t^{\gamma-1}$ is convex, hence
\[
|p_\th|^{\gamma-1}\le (1-\th)|p_0|^{\gamma-1}+\th|p_1|^{\gamma-1}.
\]
If $1<\gamma<2$, then $0<\gamma-1<1$ and $t\mapsto t^{\gamma-1}$ is subadditive, hence
\[
|p_\th|^{\gamma-1}\le (1-\th)^{\gamma-1}|p_0|^{\gamma-1}+\th^{\gamma-1}|p_1|^{\gamma-1}.
\]
In both cases,
\[
|p_\th|^{\gamma-1}\le (1-\th)^{\min(\gamma-1,1)}|p_0|^{\gamma-1}+\th^{\min(\gamma-1,1)}|p_1|^{\gamma-1}.
\]
    This yields the result.   
\end{proof}

The final result in this section is a composite estimate that captures the interplay between \( g^{-1} \), \( H\), and \( D_pH \) under convex combinations. This estimate is also key in establishing the continuity of the MFG operator.

\begin{lemma}\label{lem:boundOng^-1DpH1}
    Suppose that Assumptions~\ref{assume:D.1} and \ref{assume:S.3}--\ref{assume:S.5} hold. Let $p_1, p_0 \in \R^d$ be arbitrary. Set $p_\th = (1-\th)p_0 + \th p_1 $. 
    Then, for a.e.\ $x\in\Omega$ and every $\th \in [0,1]$, we have
\[
    g^{-1}(H(x,p_\th)) \, |D_pH(x, p_\th)| \leq C \left( \th^{\min(\gamma-1, 1)} |p_1|^{\gamma-1} +
    (1-\th)^{\min(\gamma-1, 1)}  |p_0|^{\gamma-1}  +1 \right)
\]
for some constant $C>1$ independent of $p_0$, $p_1$,  and $\th$.
\end{lemma}
\begin{proof}
    From Assumption~\ref{assume:S.4} and Corollary~\ref{cor:boundOng^-1(H)1}, we have the bounds $|D_pH(x, p_\th)| \le C(|p_\th|^{\alpha-1} + 1)$ and $g^{-1}(H(x,p_\th)) \le C(|p_\th|^{\alpha/\beta} + 1)$. Multiplying these implies
    \[
    g^{-1}(H(x,p_\th)) |D_pH(x, p_\th)| \le C(|p_\th|^{\gamma-1} + 1),
    \]
since $\gamma-1=(\a-1)+\a/\b$, so $\a-1\le \gamma-1$ and $\a/\b\le \gamma-1$, and therefore
\[
(|p_\th|^{\a-1}+1)(|p_\th|^{\a/\b}+1)
\le |p_\th|^{\gamma-1}+|p_\th|^{\a-1}+|p_\th|^{\a/\b}+1
\le C(|p_\th|^{\gamma-1}+1).
\]
 We bound $|p_\th|^{\gamma-1}$ using the triangle inequality $|p_\th| \le (1-\th)|p_0| + \th|p_1|$.
    \begin{itemize}[leftmargin=*]
        \item If $\gamma \ge 2$, then $\gamma-1 \ge 1$ and $t \mapsto t^{\gamma-1}$ is convex. Thus,
        \[ |p_\th|^{\gamma-1} \le (1-\th)|p_0|^{\gamma-1} + \th|p_1|^{\gamma-1}. \]
        \item If $1 < \gamma < 2$, then $0 < \gamma-1 < 1$ and $t \mapsto t^{\gamma-1}$ is subadditive. Thus,
        \[ |p_\th|^{\gamma-1} \le ((1-\th)|p_0|)^{\gamma-1} + (\th|p_1|)^{\gamma-1} = (1-\th)^{\gamma-1}|p_0|^{\gamma-1} + \th^{\gamma-1}|p_1|^{\gamma-1}. \]
    \end{itemize}
    Combining these cases, we have $|p_\th|^{\gamma-1} \le (1-\th)^{\min(1, \gamma-1)}|p_0|^{\gamma-1} + \th^{\min(1, \gamma-1)}|p_1|^{\gamma-1}$, which yields the result.
\end{proof}

\section{The MFG Operator} \label{sec:MonoOpSepCase}

In this section, we construct the operator $ A_\epsilon $ associated with MFG System~\ref{mfg:2}. We first motivate its structure by formally deriving it from a penalized variational principle. Subsequently, we establish the functional framework and rigorously define the operator via a duality pairing on the product space $\cX$, which allows us to incorporate the mixed boundary conditions.

We recall that \( \Omega \subset \mathbb{R}^d \) is a bounded, open domain with a \( C^1 \) boundary \( \Gamma := \partial \Omega \). The boundary is partitioned as in MFG~System~\ref{mfg:2} into two disjoint subsets \( \Gamma_D \) and \( \Gamma_N \), with \( \Gamma = \overline{\Gamma_D \cup \Gamma_N} \) and \( \Gamma_D,\Gamma_N \neq \emptyset \). Throughout this section, we work under Assumptions~\ref{assumset:D} and~\ref{assumset:S}.

\subsection{Variational Derivation and Interpretation} \label{subsec:DerivationOfMonoOp}

In this subsection, we derive the penalized MFG operator \( A_\epsilon \) from a variational principle and link it to the strong formulation of System~\ref{mfg:2}. This approach synthesizes the variational formulation for mixed boundary conditions developed in \cite{Alharbi2026MFG} with the duality techniques found in \cite{Card1order}.

\paragraph{The Penalized Functional.}
Let \( g^{-1} \) be the extended inverse of \( g \) (Definition~\ref{def:extendedInverse}) and let \( G: \R \to \R \) be its primitive. We begin with the unpenalized baseline problem: minimizing \( \int_{\Omega} G(H(x,Du)) \dx - \int_{\Gamma_N} ju \ds \) subject to \( u \le 0 \) on \( \Gamma_D \).

To explicitly incorporate the density \( m \), we utilize the Fenchel conjugate \( G^*(m) = \sup_{z} \{mz - G(z)\} \). For \( m \ge 0 \), this satisfies
\[
    G^*(m) = m g(m) - G(g(m)),
\]
allowing us to write \( G(H) = \sup_{m \ge 0} \{ mH - G^*(m) \} \).
Similarly, to enforce the boundary inequality constraint \(u\le 0\) on \(\Gamma_D\), we introduce the penalty term \( \frac{1}{\gamma\epsilon^\gamma} \int_{\Gamma_D} u_+^{\gamma} \ds \).
 Using convex duality on \( L^\gamma(\Gamma_D) \), we express this penalty as a maximization over the boundary flux \( h \):
\[
    \frac{1}{\gamma\epsilon^\gamma} \int_{\Gamma_D} u_+^{\gamma} \ds = \sup_{\substack{h \in L^{\gamma'}(\Gamma_D) \\ h \ge 0}} \int_{\Gamma_D} \left( hu - \frac{\epsilon^{\gamma'}}{\gamma'} h^{\gamma'} \right) \ds.
\]
Substituting these relations into the baseline problem 
leads to the following minimax problem on the convex cone of nonnegative densities and fluxes (which we will rigorously define as $\cX^+$ in the next subsection):
\begin{equation} \label{eq:DerSepA3}
    \inf_{u}  \sup_{h\geq 0,m \geq 0} \enspace \tilde{\cI} (m,u,h),
\end{equation}
where
\begin{equation} \label{eq:DefOfcI}
    \begin{split}
        \tilde{\cI} (m,u,h) := \int_{\Omega} \big( m H(x,Du) - m g(m) &+ G(g(m))\big) \dx \\
        &- \int_{\Gamma_N} ju  \ds + \int_{\Gamma_D} \left(hu-\frac{\epsilon^{\gamma'}}{\gamma'} h^{\gamma'}\right) \ds.
    \end{split}
\end{equation}

\paragraph{First Variations and Physical Consistency.}
The functional \( \tilde{\cI} \) is convex in \( u \) and concave in \( (m,h) \) and the operator \( A_\epsilon \) corresponds to the skew-symmetric gradient-like map of this saddle-point problem:
\[
    (m,u,h) \mapsto \left(-\frac{\delta \tilde{\cI}}{\delta m}, \, \frac{\delta \tilde{\cI}}{\delta u}, \, -\frac{\delta \tilde{\cI}}{\delta h}\right).
\]
We note that while the G\^ateaux derivatives below are computed formally for smooth triplets, the resulting identities extend to \( \cX^+ \) via the growth bounds in Assumptions~\ref{assumset:S} and the dominated convergence theorem (see \cite{Alharbi2026MFG}). This saddle-point gradient map is the monotone operator studied in Subsection~\ref{subsec:MonotonicityProofSepA}.

\begin{enumerate}
    \item[\bfseries I.] {\bfseries Variation in \( u \):} For any \( v \in W^{1,\gamma}(\Omega) \), the stationarity condition \( \frac{\delta \tilde{\cI}}{\delta u} = 0 \) yields
    \begin{equation}\label{eq:Derof1stVaruSep1}
        \int_{\Omega}  m D_pH(x,Du)\cdot Dv \dx - \int_{\Gamma_N} jv  \ds + \int_{\Gamma_D} hv \ds = 0.
    \end{equation}
    Formally, this corresponds to the conservation law \( -\Div(m D_pH(x,Du)) = 0 \) with boundary fluxes \( mD_pH \cdot \nu = j \) on \( \Gamma_N \) and \( -mD_pH \cdot \nu = h \) on \( \Gamma_D \).

    \item[\bfseries II.] {\bfseries Variation in \( m \):} For any \( \mu \ge 0 \), the optimality condition \( \langle -\frac{\delta \tilde{\cI}}{\delta m}, \mu-m \rangle \ge 0 \) yields
    \begin{equation} \label{eq:Derof1stVarmSep1}
        \int_{\Omega} (\mu-m) (-H(x,Du)+g(m)) \dx \ge 0.
    \end{equation}
    This encodes the full complementarity system:
    \[
        m\ge 0,\qquad -H(x,Du)+g(m)\ge 0,\qquad m\bigl(-H(x,Du)+g(m)\bigr)=0
        \quad\text{a.e. in }\Omega.
    \]
    This is equivalent to the explicit inversion \( m = g^{-1}(H(x,Du)) \).

    \item[\bfseries III.] {\bfseries Variation in \( h \):} For any \( k \ge 0 \), the optimality condition \( \langle -\frac{\delta \tilde{\cI}}{\delta h}, k-h \rangle \ge 0 \) yields
    \begin{equation} \label{eq:Derof1stVarhSep1}
        \int_{\Gamma_D} (k-h)(-u+\epsilon^{\gamma'}h^{\gamma'-1})\ds \ge 0.
    \end{equation}
    This implies the penalized Signorini conditions:
    \[
        h \ge 0, \quad \epsilon^{\gamma'}h^{\gamma'-1} \ge u, \quad \text{and} \quad h(\epsilon^{\gamma'}h^{\gamma'-1}-u)=0.
    \]
    Solving for \( h \), we obtain \( h = \epsilon^{-\gamma} u_+^{\gamma-1} \), equivalently,
    \begin{equation}\label{eq:SubsecRemarks:hAndu_+}
        \epsilon^\gamma\,h = u_+^{\gamma-1}.
    \end{equation}
    Introducing \( h \) as an independent variable circumvents the difficulty inherent in the nonlinearity of the flux \( J := -mD_pH \) by providing a linear variable in the boundary integral. 
    Moreover, using the boundary flux relation \(h=-mD_pH(x,Du)\cdot \nu\) on \(\Gamma_D\) (formally for smooth triplets), \eqref{eq:SubsecRemarks:hAndu_+} yields the boundary identity:
    \begin{equation} \label{eq:SubsecRemarks:contact-set}
        -u\, mD_pH(x,Du)\cdot \nu =  \epsilon^{-\gamma} \, u_+^{\gamma} \qquad \text{ on } \Gamma_D.
    \end{equation}
\end{enumerate}

\paragraph{The Variational Inequality.}
Combining these conditions, any saddle point \( (m,u,h) \) of \eqref{eq:DerSepA3} satisfies the variational inequality:
\begin{equation}\label{eq:varIneq}
    \left\langle A_{\epsilon} \begin{bmatrix} m \\ u \\ h \end{bmatrix}, \begin{bmatrix} \mu-m \\ v-u \\ k-h\end{bmatrix} \right \rangle \geq 0 \qquad \text{for all } (\mu, v, k) \in \cX^+.
\end{equation}
Formally, \eqref{eq:varIneq} recovers System~\ref{mfg:2}: the $u$-variation yields $-\Div(mD_pH(x,Du))=0$ in $\Omega$ with fluxes $mD_pH\cdot \nu=j$ on $\Gamma_N$ and $-mD_pH\cdot \nu=h$ on $\Gamma_D$; the $m$-variation yields $m=g^{-1}(H(x,Du))$ a.e.\ in $\Omega$; and the $h$-variation yields the penalized Signorini condition on $\Gamma_D$, so that if $u<0$ on $\Gamma_D$ then $h=0$ and the outflow vanishes.

\begin{remark}[Regularization and Domain]
    The parameter \( \epsilon > 0 \) acts as a boundary regularization, introducing the term \( h^{\gamma'} \) which ensures coercivity in \( h \). Recovering the original hard constraint \( u \le 0 \) requires the limit \( \epsilon \to 0^+ \), using estimates derived from \eqref{eq:SubsecRemarks:contact-set}.
    Furthermore, while the penalized operator \( A_\epsilon \) requires \( h \in L^{\gamma'}(\Gamma_D) \) to define \( h^{\gamma'-1} \), the limit operator \( A_0 \) admits \( h \) in the strictly larger dual trace space \( W^{-1+\frac{1}{\gamma},\gamma'}(\Gamma_D) \), consistent with interpreting the limit boundary flux as a distribution.
\end{remark}

\subsection{Functional Framework} \label{subsec:FuncFramework}

We first establish the Banach spaces and duality pairings necessary for the weak formulation. Let the exponents \(\beta\) and \(\gamma\) be given as in Assumption~\ref{assume:D.1}, and let \(\gamma'\) denote the H\"older conjugate of \(\gamma\). We define the product space
\[
    \cX:= L^{\beta+1}(\Omega)\times W^{1,\gamma}(\Omega) \times L^{\gamma'}(\Gamma_D).
\]
The topological dual of \(\cX\) can be identified with
\[
    \cX' = L^{\frac{\beta+1}{\beta}}(\Omega) \times  W^{-1,\gamma'}(\Omega) \times L^{\gamma}(\Gamma_D),
\]
where \(W^{-1,\gamma'}(\Omega)\) denotes the dual space \((W^{1,\gamma}(\Omega))'\).

\paragraph{The Operator Domain.}
The domain of our operator is the convex cone of nonnegative densities and boundary fluxes:
\begin{equation} \label{eq:defOfcX^+}
    \cX^+ := \big\{ (m, u, h) \in \cX : m \geq 0 \text{ a.e. in } \Omega \text{ and } h \geq 0 \text{ a.e. on } \Gamma_D \big\}.
\end{equation}

\paragraph{Trace Spaces and Duality.}
While the \(L^p\) components of \(\cX'\) are standard, the handling of the divergence terms requires precise properties of the Sobolev dual. Let \( T: W^{1,\gamma}(\Omega) \to W^{1-\frac{1}{\gamma}, \gamma}(\Gamma) \) denote the continuous trace operator, and let \( W^{-1+\frac{1}{\gamma},\gamma'}(\Gamma) \) denote the dual of the trace space. We rely on the continuous embedding (see, e.g., \cite{DiBenedetto1}):
\begin{equation} \label{eq:dualityEmbedding1}
    L^{\gamma'}(\Omega) \times L^{\gamma'}(\Omega; \R^d) \times W^{-1+\frac{1}{\gamma},\gamma'}(\Gamma) \hookrightarrow W^{-1,\gamma'}(\Omega).
\end{equation}
This embedding identifies a triplet \( \mathbf{\Xi} = (f, F, \mathfrak{f}) \) in the product space with a functional acting on \( \varphi \in W^{1,\gamma}(\Omega) \) via the pairing:
\[
    \langle \mathbf{\Xi}, \varphi \rangle_{W^{-1,\gamma'}, W^{1,\gamma}}
    = \int_{\Omega} f\, \varphi \dx
    + \int_{\Omega} F \cdot D\varphi \dx
    + \big\langle \mathfrak f, T\varphi \big\rangle_{\Gamma}.
\]
Here, \(\langle \cdot, \cdot \rangle_{\Gamma}\) denotes the duality pairing between \(W^{-1+\frac{1}{\gamma},\gamma'}(\Gamma)\) and the trace space; when \(\mathfrak f \in L^{\gamma'}(\Gamma)\), this reduces to the standard boundary integral \(\int_{\Gamma} \mathfrak f \,T\varphi \ds\).

In the context of the operator \(A_\epsilon(m,u,h)\), the transport component corresponds to the triplet
\[
    \bigl(0,\; mD_pH(\cdot,Du),\; -j\,\Chi_{\Gamma_N}+h\,\Chi_{\Gamma_D}\bigr)
\]
under the embedding \eqref{eq:dualityEmbedding1}. For the penalized operator \( A_\epsilon \), the boundary terms lie in \( L^{\gamma'}(\Gamma) \). However, for the limit operator \( A_0 \), the boundary flux \( h \) belongs to the wider space \( W^{-1+\frac{1}{\gamma},\gamma'}(\Gamma_D) \), requiring the boundary terms to be interpreted strictly via the trace duality.

\subsection{Definition and Well-Posedness} \label{subsec:DefiningOperator}

We define the MFG operator \( A_{\epsilon}: \cX^+ \to \cX' \) associated with System~\ref{mfg:2}. While formally \( A_{\epsilon} \) corresponds to the vector-valued operator
\begin{equation}\label{eq:SepA_formal}
    A_{\epsilon} \begin{bmatrix} m \\ u \\ h \end{bmatrix} :=
    \begin{bmatrix}
        -H(x,Du) + g(m) \\
        -\Div \big(m\,D_pH(x,Du) \big) + \big(m\,D_pH(x,Du)\cdot \nu\,\Chi_{\Gamma} - j\,\Chi_{\Gamma_N} + h\,\Chi_{\Gamma_D}\big)\cH^{d-1}\\
        \big(-u + \epsilon^{\gamma'}h^{\gamma'-1}\big)\Chi_{\Gamma_D}\cH^{d-1}
    \end{bmatrix},
\end{equation}
we define it rigorously via a duality pairing to ensure boundedness in \( \cX' \). 
In particular, the transport component (the second line in \eqref{eq:SepA_formal}) is encoded through its action on \(v\in W^{1,\gamma}(\Omega)\) by
\[
v \longmapsto \int_\Omega mD_pH(x,Du)\cdot Dv\,\dx \;-\; \int_{\Gamma_N} j\,v\,\ds \;+\; \int_{\Gamma_D} h\,v\,\ds,
\]
which agrees with the formal expression in \eqref{eq:SepA_formal} whenever \(mD_pH(x,Du)\) admits a normal trace on \(\Gamma\). Hence no pointwise normal trace of \(mD_pH(x,Du)\) is required to define \(A_\epsilon\) at this stage.

\begin{definition}[The Operator $A_\epsilon$]
For any \( (m,u,h) \in \cX^+ \) and test triplet \( (\mu, v, k) \in \cX \), we define the action of \( A_{\epsilon} \) by:
\begin{equation}\label{eq:A_rigorous_pairing}
    \begin{split}
        \left\langle A_{\epsilon} \begin{bmatrix} m \\ u \\ h \end{bmatrix}  ,   \begin{bmatrix}  \mu \\ v \\ k \end{bmatrix}  \right\rangle
        &:= \int_\Omega \bigl[-H(x,Du)+g(m)\bigr] \mu \dx
        + \int_\Omega mD_pH(x,Du) \cdot Dv \dx \\
        &\quad - \int_{\Gamma_N} jv \ds + \int_{\Gamma_D} h v \ds
        + \int_{\Gamma_D} (-u+\epsilon^{\gamma'}h^{\gamma'-1}) k \ds.
    \end{split}
\end{equation}
In this expression, \( v \) and \( u \) are understood in the sense of traces on \( \Gamma \) and \( \Gamma_D \), respectively.
\end{definition}

We now verify that this pairing defines a bounded mapping.

\begin{proposition}\label{prop:WellDefinedness}
    Under Assumptions~\ref{assumset:S} and \ref{assumset:D}, the operator \( A_{\epsilon} : \cX^+ \to \cX' \), defined by \eqref{eq:A_rigorous_pairing}, is well-defined for any \( \epsilon > 0 \).
\end{proposition}

\begin{proof}
    To establish that \( A_{\epsilon}(\cX^+ ) \subseteq \cX' \), it suffices to demonstrate that each component of the pairing defines a functional in the appropriate space. Specifically, we show:
    \[
    \begin{array}{rlcrl}
        -H(\cdot,Du) + g(m) &\in L^{\frac{\beta+1}{\beta}}(\Omega), &  &
        m D_p H(\cdot,Du) &\in L^{\gamma'}(\Omega; \mathbb{R}^d), \\[0.5em]
        -u|_{\Gamma_D} + \epsilon^{\gamma'} h^{\gamma'-1} &\in L^{\gamma}(\Gamma_D), & \text{ and } &
        - j \, \Chi_{\Gamma_N} + h \, \Chi_{\Gamma_D} &\in W^{-1+\frac{1}{\gamma}, \gamma'}(\Gamma).
    \end{array}
    \]
    We establish these bounds as follows:
    \begin{enumerate}
        \item \textbf{Hamiltonian and Coupling:} Assumption~\ref{assume:S.4} (specifically Remark~\ref{rmk:HamiltonImplicitBound}) implies the bound \( |H(x,p)| \le C(|p|^\alpha + 1) \). Since $\alpha\frac{\beta+1}{\beta} = \gamma$, we have \( H(\cdot,Du) \in L^{\frac{\beta+1}{\beta}}(\Omega) \). Similarly, Assumption~\ref{assume:S.3} implies \( g(m) \in L^{\frac{\beta+1}{\beta}}(\Omega) \). Thus, the sum lies in \( L^{\frac{\beta+1}{\beta}}(\Omega) \).
\item \textbf{Transport Flux:} Assumption~\ref{assume:S.4} yields \( |D_pH(x,p)| \le C(|p|^{\alpha-1} +1) \); hence,
\(
|D_pH(x,Du)|^{\gamma'} \le C(|Du|^{(\alpha-1)\gamma'}+1).
\)
Applying Young's inequality to the mixed term with conjugate exponents
\[
\lambda=\frac{\beta+1}{\gamma'},\qquad \lambda'=\frac{\gamma-1}{\alpha-1}
\quad\left(\text{so that } \frac{1}{\lambda}+\frac{1}{\lambda'}=1\right),
\]
and using \(m^{\gamma'}\le 1+m^{\beta+1}\) (since \(\gamma'\le \beta+1\) when \(\alpha>1\)), yields
\[
    m^{\gamma'}|D_pH(x,Du)| ^{\gamma'} \leq C m^{\beta+1} + C(|Du|^{\gamma} + 1).
\]
Thus, \( mD_pH(x,Du) \in L^{\gamma'}(\Omega; \mathbb{R}^d) \).

        \item \textbf{Boundary Terms:} Since \( j \in L^{\gamma'}(\Gamma_N) \) and \( h \in L^{\gamma'}(\Gamma_D) \), the boundary flux \( - j\Chi_{\Gamma_N}+ h\Chi_{\Gamma_D} \) lies in \( L^{\gamma'}(\Gamma) \subset W^{-1+\frac{1}{\gamma},\gamma'}(\Gamma) \). Finally, the trace theorem ensures \( u|_{\Gamma_D} \in L^{\gamma}(\Gamma_D) \), and \( h^{\gamma'-1} \in L^{\gamma}(\Gamma_D) \), so the penalty term is well-defined.
    \end{enumerate}
    This confirms that the pairing defines a map from \(\cX^+\) to \(\cX'\).
\end{proof}

\section{On the Properties of the MFG Operator} \label{sec:KeyPropSepCase}

In this section, we investigate the fundamental properties of the penalized MFG operator 
\( A_\epsilon \). To simplify the presentation, throughout this and the following section, we fix $\epsilon=1$ and write
\begin{equation}\label{eq:TheFullSepOpA}
    A := A_1,
\end{equation}
noting that all statements and proofs extend naturally to arbitrary $\epsilon>0$. We begin by establishing the monotonicity and hemicontinuity of the full operator $A$. Then, we introduce a restricted auxiliary operator $A_u$ to prove that the system exhibits partial coercivity with respect to $m$ and $h$. Finally, we demonstrate that $A$ lacks global coercivity due to a constant-shift degeneracy in the transport variable $u$. This failure prevents the direct application of the Browder--Minty theorem on $\cX^+$, motivating the quotient-space formulation introduced in Section~\ref{sec:SepExiste}.

\subsection{ \texorpdfstring{The monotonicity of the operator \( A \)}{The monotonicity of the operator A}} \label{subsec:MonotonicityProofSepA}

In this subsection, we establish the defining property of the monotone operator \( A \), namely, its monotonicity.

\begin{proposition} \label{prop:SepAMono}
Suppose Assumptions~\ref{assumset:S} and \ref{assumset:D} hold. Then, the operator  \( A:\cX^+ \to \cX' \), defined by \eqref{eq:TheFullSepOpA}, is monotone.
\end{proposition}
\begin{proof}
    Let \( (m,u,h), (\mu,v,k) \in \cX^+ \) be arbitrary. To establish the monotonicity of \( A \), we need to show that
    \[
    \varpi := \left\langle A \begin{bmatrix} m \\ u \\ h \end{bmatrix}  - A \begin{bmatrix} \mu \\ v \\ k\end{bmatrix}  \, ,\,  \begin{bmatrix} m \\ u \\ h \end{bmatrix}  - \begin{bmatrix} \mu \\ v \\ k\end{bmatrix}  \right\rangle \geq 0.
    \]
By rearranging the terms of the pairing, we obtain
    \[
    \begin{aligned}
        \varpi &= \int_\Omega [g(m)-g(\mu)] (m-\mu) \dx + \int_{\Gamma_D} (h^{\gamma'-1}-k^{\gamma'-1})(h-k) \ds \\
        &\quad + \int_\Omega m \left[H(x,Dv)-H(x,Du)-D_pH(x,Du) \cdot (Dv-Du) \right] \dx \\
        &\quad + \int_\Omega \mu \left[H(x,Du)-H(x,Dv)-D_pH(x,Dv) \cdot (Du-Dv) \right] \dx .
    \end{aligned}
    \]
    The first two integrals are nonnegative because the maps \(\mu \mapsto g(\mu) \) and \( h \mapsto h^{\gamma'-1} \) are increasing. The last two integrals are nonnegative due to the convexity of the map \( p \mapsto H(x, p) \) and the non-negativity of \( m \) and \( \mu \). Therefore, \( A \) is monotone.
\end{proof}

\begin{remark}[On strict monotonicity]
	Note that the operator $A$ is monotone, but not strictly monotone. The terms in the monotonicity calculation involving $H$ are guaranteed to be non-negative by convexity, but they can be zero even if $Du \neq Dv$ (for instance, if $H$ is linear in a region). 
    Even if $H$ is strictly convex, strict monotonicity in $u$ holds only \emph{modulo constants} (and can still fail in the vanishing-density regime).
    This lack of strict monotonicity in the $u$ component is a key feature of the problem. In contrast, the auxiliary operator $A_u$ in Proposition \ref{prop:SepAuxA_u1Mono} is strictly monotone in $(m,h)$ due to the strict increase of $g$ and the term $h^{\gamma'-1}$.
\end{remark}

\subsection{\texorpdfstring{The continuity of the operator \( A \)}{The continuity of the operator A}} \label{subsec:contOfSepA}

The hemicontinuity of the monotone operator \( A \) is a necessary condition for the application of the Browder--Minty theorem. 
In view of the discussion in Subsection~\ref{sec:MonoMinty}, it is enough to verify continuity along convex combinations.

\begin{proposition}\label{prop:SepA_continuity}
    Suppose Assumptions \ref{assumset:S} and \ref{assumset:D} hold, and let \( A:\cX^+ \to \cX' \) be the operator defined by \eqref{eq:TheFullSepOpA}. Let \( (m_0,u_0,h_0), (m_1,u_1,h_1) \in \cX^+ \) be fixed, and define the convex combination
    \[
        (m_\th,u_\th,h_\th):= \th (m_1,u_1,h_1) + (1-\th)(m_0,u_0,h_0)
    \]
    for \( \th \in [0,1] \). Then, for every \( (\mu, v, k) \in \cX \), the mapping
    \[
        \th \mapsto \left\langle A \begin{bmatrix} m_\th\\ u_\th \\ h_\th\end{bmatrix}  \, ,\,  \begin{bmatrix} \mu \\ v \\ k\end{bmatrix}  \right\rangle
    \]
    is continuous on \( [0,1] \).
\end{proposition}

\begin{proof}
    Fix \( (m_0,u_0,h_0), (m_1,u_1,h_1) \in \cX^+ \) and define the convex combination
    \[
        (m_\th,u_\th,h_\th) := \th (m_1,u_1,h_1) + (1-\th)(m_0,u_0,h_0),
    \]
    for \( \th \in [0,1] \). For an arbitrary \( (\mu, v, k) \in \cX \), we consider the function
    \[
    \begin{split}
    f(\th) := \left\langle A \begin{bmatrix} m_\th\\ u_\th \\ h_\th\end{bmatrix}  \, ,\,  \begin{bmatrix} \mu \\ v \\ k\end{bmatrix}  \right\rangle.
    \end{split}
    \]
Expanding this expression using the definition of \( A \), we obtain
    \[
    \begin{aligned}
        f(\th) &= \int_\Omega [-H(x,Du_\th) + g(m_\th)] \mu \dx + \int_\Omega m_\th D_pH(x,Du_\th)\cdot Dv \dx \\
        &\quad - \int_{\Gamma_N} j v \ds + \int_{\Gamma_D} [h_\th v - u_\th k] \ds + \int_{\Gamma_D} h_\th^{\gamma'-1} k \ds.
    \end{aligned}
    \]
    Our goal is to prove that \( f \) is continuous using the dominated convergence theorem.

From Lemma~\ref{lem:convProdSC1} and Young's inequality, with
$p_\th := Du_\th$ and
$p_i := Du_i$ for $i = 0,1$, we obtain   
    \[
        \begin{split}
             |-H(x, p_\th) + g(m_\th)||\mu| &\leq  C\left[|p_1|^{\alpha} + m_1^{\beta}
            +   |p_0|^{\alpha} + m_0^{\beta} +1\right]|\mu| \\
            & \leq  C\left[|p_1|^{\gamma} + m_1^{\beta+1}
            +   |p_0|^{\gamma} + m_0^{\beta+1} + |\mu|^{\beta+1} +1\right], 
        \end{split}
    \]
where we used the exponent relations in \eqref{eq:parameters1} and Young's inequality to bound the products $|p|^\alpha |\mu|$ (with the conjugate exponents $\frac{\gamma}{\alpha}$ and $\beta+1$) and $m^\beta |\mu|$ (with the exponents $\frac{\beta+1}{\beta}$ and $\beta+1$).

    Moreover, from Lemma~\ref{lem:convProdSC2} and Young's inequality, we obtain
    \[
    \begin{split}
       | m_\th D_pH(x,Du_\th)\cdot Dv| &\leq  |m_\th D_pH(x, p_\th)| |Dv| \\
       &\leq C \left[ |p_0|^{\gamma-1}+m_0^{(\beta+1)/\gamma'}
            +  |p_1|^{\gamma-1}+m_1^{(\beta+1)/\gamma'}    +1 \right] |Dv|\\
            &\leq C \left[ |p_0|^{\gamma}+m_0^{\beta+1}
            +  |p_1|^{\gamma}+m_1^{\beta+1} + |Dv|^{\gamma}+1 \right].
    \end{split}
    \]
    Similarly, Young's inequality and the fact that \( \th \leq 1 \) give us
    \[
        h_\th^{\gamma'-1}|k| \leq C(h_1^{\gamma'}+h_0^{\gamma'}+|k|^{\gamma'}).
    \]
Hence, all integrands are dominated by an $L^1(\Omega)$ or $L^1(\Gamma_D)$ function independent of $\theta$,
and pointwise convergence follows from continuity of $H$ and $D_pH$ in $p$.    
Finally, the term $\int_{\Gamma_D} [h_\th v - u_\th k] \ds$ is continuous in $\th$ by linearity, and the term $-\int_{\Gamma_N} jv \ds$ is constant. Thus, 
by the pointwise convergence of the remaining integrands as $\th \to \th_0$ and the established dominating functions, the dominated convergence theorem yields
\[
\lim_{\th \to \th_0} f(\th) = f(\th_0).
\]
This completes the proof.
\end{proof}

\subsection{The Restricted Auxiliary Operator \texorpdfstring{\( A_u \)}{Au} and Partial Coercivity} \label{subsec:TheSepOpCoercivityInmandh}

This subsection establishes the coercivity of the operator \( A \) in \( m \) and \( h \) (see Definition~\ref{def:coercivity}).
As usual, we identify $L^{(\beta+1)/\beta}(\Omega) \times L^\gamma(\Gamma_D)$ with the dual of $L^{\beta+1}(\Omega) \times L^{\gamma'}(\Gamma_D)$. 
For each \( u \in W^{1,\gamma}(\Omega) \), we define  the restricted auxiliary operator
\[
A_u: L^{\beta+1}(\Omega) \times  L^{\gamma'}(\Gamma_D) \rightarrow L^{(\beta+1)/\beta}(\Omega)\times L^{\gamma}(\Gamma_D),
\]
via the pairing
\begin{equation}\label{eq:DefOfAuxA_u1}
     \left\langle A_u \begin{bmatrix} m \\ h \end{bmatrix} \, , \, \begin{bmatrix} \mu \\ k \end{bmatrix} \right\rangle :=  \left\langle A \begin{bmatrix} m \\ u \\ h \end{bmatrix} \, ,\,  \begin{bmatrix} \mu \\ 0 \\ k \end{bmatrix} \right\rangle.
\end{equation}

Because \( A_u \) is a restriction of the full operator \( A \) evaluated at a fixed \( u \), it naturally inherits the monotonicity and hemicontinuity of $A$. Furthermore, due to the strictly increasing nature of the coupling \( g \) and the boundary penalty term, \( A_u \) is strictly monotone.
\begin{proposition} \label{prop:SepAuxA_u1Mono}
Suppose that Assumptions~\ref{assumset:S} and \ref{assumset:D} hold, and let \( u \in W^{1,\gamma}(\Omega) \) be fixed.
Then, the operator \( A_u: L^{\beta+1}(\Omega) \times L^{\gamma'}(\Gamma_D) \to L^{\frac{\beta+1}{\beta}}(\Omega) \times L^{\gamma}(\Gamma_D) \), defined by \eqref{eq:DefOfAuxA_u1}, is strictly monotone on the cone of nonnegative pairs \((m,h)\).
\end{proposition}
\begin{proof}
Let \( u \in W^{1,\gamma}(\Omega) \) be fixed, and let \( (m,h), (\mu,k) \in L^{\beta+1}(\Omega) \times L^{\gamma'}(\Gamma_D) \) be distinct pairs of nonnegative functions. To establish the strict monotonicity of \( A_u \), we need to show that
    \begin{align*}
        &\left\langle A_u \begin{bmatrix} m \\ h \end{bmatrix} - A_u \begin{bmatrix} \mu \\ k\end{bmatrix} \, ,\, \begin{bmatrix} m \\ h \end{bmatrix} - \begin{bmatrix} \mu \\ k\end{bmatrix} \right\rangle \\
        &= \int_\Omega [g(m)-g(\mu)] (m-\mu) \dx + \int_{\Gamma_D} (h^{\gamma'-1}-k^{\gamma'-1})(h-k) \ds>0.
    \end{align*}
    Since \( g \) and \( z \mapsto z^{\gamma'-1} \) are strictly increasing, the integrands are strictly positive wherever the functions differ. Since \( (m,h) \not\equiv (\mu,k) \), at least one integral is strictly positive, implying the inequality is strict. 
\end{proof}

Finally, the following proposition establishes the coercivity of \( A_u \), which is essential for solving the sub-problem in Section~\ref{sec:SepExiste}.
\begin{corollary}\label{cor:SepAuxA_u1_continuity}
    Suppose that Assumptions \ref{assumset:S} and \ref{assumset:D} hold, and let \( u \in W^{1,\gamma}(\Omega) \) be fixed. Moreover,
    let \( A_u: L^{\beta+1}(\Omega) \times L^{\gamma'}(\Gamma_D) \to L^{\frac{\beta+1}{\beta}}(\Omega) \times L^{\gamma}(\Gamma_D) \) be the operator defined by \eqref{eq:DefOfAuxA_u1}. Let \( (m_0,h_0), (m_1, h_1) \in L^{\beta+1}(\Omega) \times L^{\gamma'}(\Gamma_D) \) be fixed pairs of nonnegative functions, and define the convex combination
    \[
        (m_\th ,h_\th):= \th (m_1, h_1) + (1-\th)(m_0, h_0)
    \]
    for \( \th \in [0,1] \). Then, for every \( (\mu, k) \in L^{\beta+1}(\Omega) \times L^{\gamma'}(\Gamma_D) \), the mapping
    \[
        \th \mapsto \left\langle A_u \begin{bmatrix} m_\th \\ h_\th\end{bmatrix} \, ,\, \begin{bmatrix} \mu \\ k\end{bmatrix} \right\rangle
    \]
    is continuous on \( [0,1] \).
\end{corollary}
\begin{proof}
Since $A_u$ is a restriction of $A$ (with $u$ fixed), the result follows immediately from Proposition~\ref{prop:SepA_continuity}.
\end{proof}

The following result establishes the coercivity of \( A_u \).

\begin{proposition}\label{prop:SepACoerformandh}
Suppose that Assumptions~\ref{assumset:D} and \ref{assumset:S} hold.
    For a fixed \( u \in W^{1,\gamma}(\Omega) \), let \( A_u \) be the operator defined by \eqref{eq:DefOfAuxA_u1}. Then, for \( m \geq 0 \) and \( h \geq 0 \), we have
    \[
        \frac{1}{\|m\|_{L^{\beta+1}(\Omega)} + \|h\|_{L^{\gamma'}(\Gamma_D)}}  \left\langle A_u \begin{bmatrix} m \\ h \end{bmatrix}-A_u \begin{bmatrix} 0 \\ 0 \end{bmatrix} \, ,\,  \begin{bmatrix} m \\ h \end{bmatrix} \right\rangle \longrightarrow \infty
    \]
    as \( \|m\|_{L^{\beta+1}(\Omega)}+\|h\|_{L^{\gamma'}(\Gamma_D)} \to \infty \).
\end{proposition}
\begin{remark}[On the assumptions in Proposition~\ref{prop:SepACoerformandh}]
Proposition~\ref{prop:SepACoerformandh} is stated under the standing hypotheses (in particular, Assumptions~\ref{assumset:S}, together with the exponent specification in Assumption~\ref{assume:D.1}) so that
\[
A_u: L^{\beta+1}(\Omega)\times L^{\gamma'}(\Gamma_D)\;\longrightarrow\; L^{(\beta+1)/\beta}(\Omega)\times L^{\gamma}(\Gamma_D)
\]
is a well-defined operator. However, the \emph{coercivity estimate itself} uses only the superlinearity of the coupling and the boundary penalty.
More precisely, we see from the proof below that once the pairing \eqref{eq:DefOfAuxA_u1} is meaningful, the proof of coercivity relies only on Assumption~\ref{assume:S.3} and the fact that $\beta>0$ (hence $\beta+1>1$) and $\gamma'>1$.
In particular, no structural assumption on $H$ and no assumption on the inflow $j$ is used in this coercivity argument.
Moreover, the constants in the coercivity bound are independent of the fixed $u$, since the terms involving $H(x,Du)$ and $u|_{\Gamma_D}$ cancel in the difference
$A_u(m,h)-A_u(0,0)$.
\end{remark}
\begin{proof}
    Fix \( u \in W^{1,\gamma}(\Omega) \), and consider arbitrary nonnegative functions \( m \in L^{\beta+1}(\Omega) \) and \( h \in L^{\gamma'}(\Gamma_D) \).
From Assumption~\ref{assume:S.3}, 
 $m(g(m) - g(0)) \geq c m^{\beta+1} - C(m+1)$. 
Accordingly, 
we obtain
    \begin{equation} \label{eq:SepACoer1Proof1}
    \begin{split}
    \left\langle A_u \begin{bmatrix} m \\ h \end{bmatrix}- A_u \begin{bmatrix} 0 \\ 0 \end{bmatrix}\, ,\, \begin{bmatrix} m \\ h \end{bmatrix} \right\rangle
        &= \int_\Omega m \bigl(g(m)-g(0)\bigr) \dx + \int_{\Gamma_D} h^{\gamma'} \ds\\
        &\geq C^{-1} \|m\|_{L^{\beta+1}(\Omega)}^{\beta+1} - C(\|m\|_{L^{\beta+1}(\Omega)} + 1) + \|h\|_{L^{\gamma'}(\Gamma_D)}^{\gamma'}.
    \end{split}
    \end{equation}
     for some \( C > 1 \), independent of \( m \) and \( h \).
    %Let \( \rho = \min(\beta+1, \gamma') \). By Assumption~\ref{assumset:D}, \(\beta > 0\) and \(\gamma > 1\), which implies \(\rho > 1\). Using the inequality \(t^p \ge t^\rho - 1\) for \(t \ge 0\) and \(p \ge \rho\), followed by the convexity inequality \(a^\rho + b^\rho \ge 2^{1-\rho}(a+b)^\rho\) for \(a,b \ge 0\), we deduce that
    % \[
    %     \|m\|_{L^{\beta+1}(\Omega)}^{\beta+1}+\|h\|_{L^{\gamma'}(\Gamma_D)}^{\gamma'} \ge \|m\|_{L^{\beta+1}(\Omega)}^{\rho}+\|h\|_{L^{\gamma'}(\Gamma_D)}^{\rho} - 2 \ge 2^{1-\rho} \bigl( \|m\|_{L^{\beta+1}(\Omega)} + \|h\|_{L^{\gamma'}(\Gamma_D)} \bigr)^\rho - 2.
    % \]
    % Substituting this into \eqref{eq:SepACoer1Proof1}, we obtain
    Thus
    \[
    \begin{split}
    \frac{\left\langle A_u \begin{bmatrix} m \\ h \end{bmatrix}- A_u \begin{bmatrix} 0 \\ 0 \end{bmatrix}\, ,\, \begin{bmatrix} m \\ h \end{bmatrix} \right\rangle}{\|m\|_{L^{\beta+1}(\Omega)}+\|h\|_{L^{\gamma'}(\Gamma_D)}}
    \to \infty,
    % &\geq \frac{C^{-1} 2^{1-\rho} \bigl( \|m\|_{L^{\beta+1}(\Omega)} + \|h\|_{L^{\gamma'}(\Gamma_D)} \bigr)^\rho - C'(\|m\|_{L^{\beta+1}(\Omega)}+1)}{\|m\|_{L^{\beta+1}(\Omega)}+\|h\|_{L^{\gamma'}(\Gamma_D)}}.
    \end{split}
    \]
as \( \|m\|_{L^{\beta+1}(\Omega)}+\|h\|_{L^{\gamma'}(\Gamma_D)} \to \infty \), concluding the proof.
\end{proof}

\subsection{Constant-Shift Degeneracy and the Failure of Coercivity} \label{subsec:failOfCoercivity}

In this subsection, we demonstrate that the operator \( A \) fails to be coercive with respect to the variable \( u \). This failure prevents the direct application of the Browder--Minty theorem.
We identify two specific cases where coercivity breaks down. To this end, we fix \( m \) and \( h \) and introduce the auxiliary operator
\[
    A_{m,h}: W^{1,\gamma}(\Omega) \to W^{-1,\gamma'}(\Omega),
\]
which we define via the restriction of the pairing as follows:
\begin{equation}\label{eq:DefOfAuxA_mh}
     \left\langle A_{m,h} u  \, , \, v \right\rangle :=  \left\langle A \begin{bmatrix} m \\ u \\ h \end{bmatrix}\, ,\,  \begin{bmatrix} 0 \\ v \\ 0 \end{bmatrix} \right\rangle.
\end{equation}

\paragraph{The vanishing density case:} Under the condition $m \equiv 0$, we observe that the interaction term in the pairing drops out, leaving
\[
\langle A_{0,h} u , u \rangle = - \int_{\Gamma_N} j u \ds + \int_{\Gamma_D} h u \ds.
\]
Applying H\"older's inequality and the trace inequality, we obtain
\[
\left|\langle A_{0,h} u , u \rangle\right| = \left|- \int_{\Gamma_N} ju \ds + \int_{\Gamma_D} h u \ds\right| \leq C\big(\|j\|_{L^{\gamma'}(\Gamma_N)}+ \|h\|_{L^{\gamma'}(\Gamma_D)}\big) \|u\|_{W^{1,\gamma}(\Omega)},
\]
where \( C \) depends only on \( \gamma \) and \( \Omega \). This linear upper bound implies that the ratio \( \langle A_{0,h} u, u \rangle / \|u\|_{W^{1,\gamma}(\Omega)} \) remains bounded, which violates the coercivity condition (Definition~\ref{def:coercivity}).

\paragraph{The flux compatibility case:}
When the boundary fluxes balance, i.e., \( \int_{\Gamma_D} h \ds = \int_{\Gamma_N} j \ds \).
This case is significant because this identity is a necessary condition for the existence of solutions (see Subsection~\ref{subsec:DomainOfOpSepCase}).
In this setting, the duality pairing is invariant under simultaneous constant shifts. Since \(D(u+c_1)=Du\) and \(D(v+c_2)=Dv\), the interior integral term remains unchanged.
 Consequently, for any constants \( c_1, c_2 \in \R \), the operator \( A_{m,h} \) satisfies
\[
\langle A_{m,h} (u + c_1), v + c_2 \rangle = \langle A_{m,h} u, v \rangle
- c_2 \left( \int_{\Gamma_N} j \ds - \int_{\Gamma_D} h \ds\right) = \langle A_{m,h} u, v \rangle.
\]
Specifically, setting \( v = u \) and \( c_1 = c_2 = c \), we obtain
\[
\left\langle A_{m,h} (u + c), u + c \right\rangle = \left\langle A_{m,h} u, u \right\rangle.
\]
This invariance under constant shifts allows $c \to \infty$, driving $\|u + c\|_{W^{1,\gamma}(\Omega)} \to \infty$ while $\langle A_{m,h}(u + c), u + c \rangle$ remains unchanged.
Thus, \( A_{m,h} \) cannot be coercive. Hence, the operator \( A \) fails to be coercive as well. In Section~\ref{sec:SepExiste}, we resolve this lack of coercivity by reformulating the problem on a quotient space.

\section{Existence of Solutions} \label{sec:SepExiste}

In this section, we establish the existence of solutions to the penalized variational inequality \eqref{eq:varIneq}.
As in the previous section, we work
in the normalized case $\epsilon=1$ (so $A=A_1$). The passage to the limit $\epsilon \to 0^+$, which yields solutions to the original MFG system, is addressed in Section~\ref{sec:backToMFG}.
The primary obstruction to a direct application of the Browder--Minty theorem is the operator's lack of coercivity in the variable \( u \), stemming from the system's invariance to constant shifts (see Subsection~\ref{subsec:failOfCoercivity}). We circumvent this issue by first solving for the pair \( (m_u, h_u) \) for each fixed \( u \), thereby isolating the coercivity failure. We then restore coercivity by reformulating the problem on a quotient space.

The proof proceeds in the following steps.
\begin{enumerate}[label=\roman*.]
    \item Establish the existence and uniqueness of \( (m_u, h_u) \) for each fixed \( u \).
    \item Modify the operator's domain to circumvent the coercivity failure in \( u \).
    \item Redefine the operator on the new domain, incorporating the solutions \( (m_u, h_u) \).
    \item Prove the coercivity of the modified operator and establish the existence of solutions.
\end{enumerate}

\subsection{Solvability of the Auxiliary Problem for $(m, h)$ }
\label{subsec:SolvingFormAndh}

In Subsection~\ref{subsec:TheSepOpCoercivityInmandh}, we established a partial coercivity result for the operator \( A \) and introduced the auxiliary operator \( A_u \), defined for fixed \( u \) through the restricted pairing \eqref{eq:DefOfAuxA_u1}.
In this subsection, we fix the function \( u \) and exploit this partial coercivity to establish the existence of solutions for the variational inequality:
\begin{equation}\label{eq:variationalIneqForSepA_u1}
	\left\langle A_u \begin{bmatrix} m \\ h \end{bmatrix}\, ,\,  \begin{bmatrix} \mu - m \\ k - h \end{bmatrix} \right\rangle \geq 0 \qquad \forall \, (\mu, k) \in L^{\beta+1}(\Omega) \times  L^{\gamma'}(\Gamma_D) \text{ with } \mu, k \geq 0.
\end{equation}
The solution we obtain is unique and denoted by \( (m_u, h_u) \in L^{\beta+1}(\Omega) \times  L^{\gamma'}(\Gamma_D) \), as shown below.  

\begin{theorem}\label{thm:Existence_of_m_and_h}
    Suppose that Assumptions~\ref{assumset:S} and \ref{assumset:D} hold, and let \( u \in W^{1,\gamma}(\Omega) \) be fixed. 
    Let \(\mathcal{K}:=\{(m,h) \in L^{\beta+1}(\Omega) \times L^{\gamma'}(\Gamma_D) : m \geq 0,\; h \geq 0\}\).
Then, there exists a unique pair \( (m_u,h_u) \in \mathcal{K} \) that solves the variational inequality \eqref{eq:variationalIneqForSepA_u1}.
\end{theorem}
\begin{proof}
Due to Proposition~\ref{prop:SepACoerformandh}, Proposition~\ref{prop:SepAuxA_u1Mono}, and Corollary~\ref{cor:SepAuxA_u1_continuity}, the operator \( A_u \) is coercive, strictly monotone, and hemicontinuous. Therefore, we apply Lemma~\ref{lem:UnboundedMintyExistence} with the nonempty, closed, and convex set $\mathcal{K}$ to show the existence of a pair \( (m_u,h_u) \in \mathcal{K} \) that solves \eqref{eq:variationalIneqForSepA_u1}.

To prove uniqueness, let \( (m_u,h_u) \) and \( (\bar{m}_u ,\bar{h}_u ) \) be two solutions. The variational inequality implies
\begin{align*}
\left\langle A_u \begin{bmatrix} m_u \\ h_u \end{bmatrix} \, ,\, 
\begin{bmatrix} \bar m_u - m_u \\ \bar h_u - h_u \end{bmatrix} \right\rangle &\ge 0, \\
\left\langle A_u \begin{bmatrix} \bar m_u \\ \bar h_u \end{bmatrix} \, ,\, 
\begin{bmatrix} m_u - \bar m_u \\ h_u - \bar h_u \end{bmatrix} \right\rangle &\ge 0.
\end{align*}
Combining these inequalities yields
\[
\left\langle A_u \begin{bmatrix} m_u \\ h_u \end{bmatrix} - A_u \begin{bmatrix} \bar m_u \\ \bar h_u \end{bmatrix} \, ,\, 
\begin{bmatrix} m_u - \bar m_u \\ h_u - \bar h_u \end{bmatrix} \right\rangle \le 0.
\]
Since \(A_u\) is strictly monotone (Proposition~\ref{prop:SepAuxA_u1Mono}), the equality must hold, which implies \( (m_u, h_u) = (\bar{m}_u, \bar{h}_u) \).
\end{proof}

Theorem~\ref{thm:Existence_of_m_and_h} guarantees the existence and uniqueness of the pair \( (m_u, h_u) \), which satisfies the variational inequality \eqref{eq:variationalIneqForSepA_u1}. While this establishes solvability, we can gain further insight by considering specific test functions in the inequality. This leads to the following Corollary that establishes pointwise characterizations of \( m_u \) and \( h_u \).
\begin{corollary}\label{cor:FormulaFor_m_and_h}
    Suppose that Assumptions~\ref{assumset:S} and \ref{assumset:D} hold. Let \( (m_u,h_u) \in \mathcal{K} \) be the unique solution provided by Theorem~\ref{thm:Existence_of_m_and_h} to the inequality~\eqref{eq:variationalIneqForSepA_u1}. Then, the following characterizations hold:

    \begin{enumerate}[label=\roman*.]
\item\label{subcor:FormulaFor_m_and_h1} For almost every \( x \in \Omega \),
    \begin{equation*}
        -H(x,Du(x)) + g(m_u(x)) \geq 0 \quad \text{ and } \quad m_u(x)\bigl(-H(x,Du(x)) + g(m_u(x))\bigr) = 0.
    \end{equation*}
    Consequently, for almost every \( x \in \Omega \),
    \begin{equation*}
     m_u(x) = g^{-1}(H(x,Du(x))),
    \end{equation*}
    where \(g^{-1}\) is the extended inverse of \(g\) (see Definition~\ref{def:extendedInverse}).
        \item\label{subcor:FormulaFor_m_and_h2} For almost every \( x \in \Gamma_D \),
        \begin{equation*}
            -u(x) + h_u(x)^{\gamma'-1} \geq 0 \quad \text{ and } \quad h_u(x)(-u(x) + h_u(x)^{\gamma'-1} )= 0.
        \end{equation*}
        Consequently, for almost every \( x \in \Gamma_D \),
        \begin{equation*}
            h_u(x) = u_+^{\gamma-1}(x).
        \end{equation*}
    \end{enumerate}
\end{corollary}

\begin{remark}\label{rmk:ForCor:FormulaFor_m_and_h}
    In the case of arbitrary \( \epsilon \), the formula for \( m \) is the same. However, we have that
    \begin{equation}
        h_u = \frac{u_+^{\gamma-1}}{\epsilon^{\gamma}}.
    \end{equation}
\end{remark}

\begin{proof}[\textbf{Proof of Corollary~\ref{cor:FormulaFor_m_and_h}.}] Let \( (m_u,h_u) \in L^{\beta+1}(\Omega) \times L^{\gamma'}(\Gamma_D) \) be the solution  provided by Theorem~\ref{thm:Existence_of_m_and_h}. We prove each statement separately.

\paragraph{\normalfont \emph{\ref{subcor:FormulaFor_m_and_h1}}} 
Consider arbitrary test functions \( \varphi \in C_c^{\infty}(\Omega) \) and \( \phi \in C_c^{\infty}(\Omega) \) such that \( \varphi \geq 0 \) and \( \|\phi\|_{\infty} \leq 1 \).
Define
\[
\mu = \left( 1 + \phi \right) m_u + \varphi \quad \text{ and } \quad k = h_u.
\]
Because \( (m_u,h_u) \) solve the variational inequality~\eqref{eq:variationalIneqForSepA_u1}, we obtain
\[
    \begin{split}
		0 \leq \left\langle A_u \begin{bmatrix} m_u \\ h_u \end{bmatrix}  \, ,\,  \begin{bmatrix}  \phi  \, m_u + \varphi \\ 0 \end{bmatrix}  \right\rangle
        = \int_\Omega (-H(x,Du)+g(m_u))\left( \phi\, m_u + \varphi \right) \dx.
	\end{split}
\]

Because \( \varphi \) is arbitrary with \(\varphi\ge 0\), taking \(\phi\equiv 0\) yields \(-H(x,Du)+g(m_u)\ge 0\) a.e.\ in \(\Omega\).
Then, taking \(\varphi\equiv 0\) and \(\phi=\pm \psi\) with \(0\le \psi\le 1\) gives
\(\int_\Omega (-H(x,Du)+g(m_u))\,\psi\,m_u\,\dx=0\). Hence, \(m_u(-H(x,Du)+g(m_u))=0\) a.e.\ in \(\Omega\).
Taking into account that either \( m_u = 0\) or \( m_u >0 \),  using the definition of the extended inverse, we have
\begin{enumerate}[label=-]
\item If \( m_u = 0\), then \( H(x,Du) \leq g(0) \). Therefore, \( g^{-1}(H(x,Du)) = 0 \).
    \item If \( m_u >0 \), then \( H(x,Du) = g(m_u) \). Therefore, \( g^{-1}(H(x,Du)) = m_u \).
\end{enumerate}

\paragraph{\normalfont \emph{\ref{subcor:FormulaFor_m_and_h2}}} 
Similar to the above, consider arbitrary test functions \( \varphi, \phi \in C_c^{\infty}(\Gamma_D) \)
such that \( \varphi \geq 0 \) and \( \|\phi\|_{\infty} \leq 1 \).
Define
\[
k = \left(1+\phi\right) h_u + \varphi, \quad \text{and set } \mu = m_u.
\]
Because \( (m_u,h_u) \) solve the variational inequality~\eqref{eq:variationalIneqForSepA_u1}, we obtain
\[
\begin{split}
		0 \leq \left\langle A_u \begin{bmatrix} m_u \\ h_u \end{bmatrix}  \, ,\,  \begin{bmatrix} 0 \\  \phi\, h_u + \varphi \end{bmatrix}  \right\rangle = \int_{\Gamma_D} (-u+h_u^{\gamma'-1})(\phi\, h_u + \varphi) \ds.
\end{split}
\]
Because \( \varphi \) is arbitrary with \(\varphi\ge 0\), taking \(\phi\equiv 0\) yields \(-u+h_u^{\gamma'-1}\ge 0\) a.e.\ on \(\Gamma_D\).
Then, taking \(\varphi\equiv 0\) and \(\phi=\pm \psi\) with \(0\le \psi\le 1\) gives
\(\int_{\Gamma_D} (-u+h_u^{\gamma'-1})\,\psi\,h_u\,\ds=0\). Hence, \(h_u(-u+h_u^{\gamma'-1})=0\) a.e.\ on \(\Gamma_D\). Moreover, 
taking into account that  either \( h_u = 0\) or \( h_u >0 \), we have
\begin{enumerate}[label=-]
    \item If \( h_u = 0\), then \( 0 \geq u  \). Therefore,
    \[
       h_u =  0 = u_+^{\gamma-1}.
    \]
    \item If \( h_u >0 \), then \( h_u^{\gamma'-1} = u \). Therefore,
    \[
        h_u = u_+^{\gamma-1}.
    \]
\end{enumerate}
Thus, in both cases, we conclude that
\[
h_u = u_+^{\gamma-1}.
\]
This completes the proof.
\end{proof}

\subsection{The modified domain} \label{subsec:DomainOfOpSepCase}

The failure of coercivity identified in Subsection~\ref{subsec:failOfCoercivity} prevents a direct application of the Browder--Minty theorem on $\cX^+$. This failure is structurally linked to the operator's insensitivity to additive constants in \( u \) when boundary fluxes are balanced. To address this, we define the operator on the quotient space of functions modulo constants.
Specifically, we define the quotient space
\begin{equation} \label{eq:QuotientSpace}
    \cQ = \faktor{W^{1,\gamma}(\Omega)}{\{ u = c \colon c \in \R\}},
\end{equation}
where each equivalence class \( \bar{u} \in \cQ \) is represented by an element \( u \in \bar{u} \).
The quotient space \(\cQ\) 
is reflexive and admits a norm equivalent to $\|Du\|_{L^\gamma}$;
we establish these properties in the following lemma.
\begin{lemma}\label{lem:PropertiesOfcQ}
    The quotient space \( \cQ \) defined in \eqref{eq:QuotientSpace} has the following properties:
    \begin{enumerate}[label=\roman*.]
        \item\label{Sublem:PropertiesOfcQ1} The space \( \cQ \) is reflexive, and its dual space consists precisely of functionals that annihilate constants, i.e.,
        \[
            \cQ' \cong \{ v \in W^{-1,\gamma'}(\Omega) \colon \langle v , u \rangle = 0 \text{ for all } u \equiv c, \, c \in \R \}.
        \]
        \item\label{Sublem:PropertiesOfcQ2} The seminorm
        \begin{equation}\label{eq:QuotientSpaceNorm1}
        \|\bar{u}\|_{\cQ} := \|Du\|_{L^{\gamma}(\Omega)}
        \end{equation}
        defines a norm on \( \cQ \) equivalent to its quotient norm.
    \end{enumerate}
\end{lemma}
\begin{proof} \emph{\ref{Sublem:PropertiesOfcQ1}} The reflexivity of \( \cQ \) follows from classical functional analysis results (see \cite{megginson1998introduction}, Corollary 1.11.18 and Theorem 1.10.16).

\emph{\ref{Sublem:PropertiesOfcQ2}} The quotient norm of \( \cQ \) is 
\[
    \|\bar{u}\|_{in,\cQ} := \inf_{c \in \R} \|u-c\|_{W^{1,\gamma}(\Omega)}
    = \|Du\|_{L^{\gamma}(\Omega)} + \inf_{c \in \R} \|u-c\|_{L^{\gamma}(\Omega)}.
\]
Setting
\begin{equation}\label{eq:avg}
    c = \avg_{\Omega} u := \frac{1}{|\Omega|} \int_{\Omega} u(x)\dx,
\end{equation}
        and using the Poincar{\'e} inequality, we obtain
    \begin{equation}\label{eq:quitentPonicareIneq}
    \begin{split}
         \|Du\|_{L^{\gamma}(\Omega)} \leq \|\bar u\|_{in,\cQ} \leq \|Du\|_{L^{\gamma}(\Omega)} + \|u-\avg_{\Omega}u\|_{L^{\gamma}(\Omega)} &\leq C \|Du\|_{L^{\gamma}(\Omega)},
    \end{split}
    \end{equation}
where \( C > 1 \) depends only on \( \gamma \) and the domain \( \Omega \). This shows that \(\|\bar{u}\|_{\cQ} := \|Du\|_{L^{\gamma}(\Omega)}\) is an equivalent norm.
\end{proof}

Before we specify the domain modification that affects the function \( h \), we first establish that solutions belong to a specific subset of the function space, allowing us to restrict the operator’s domain without excluding relevant solutions.
More precisely, consider a triplet \((m, u, h)\) that solves \eqref{eq:varIneq}. Fix \(c \in \R\) and \(v \in W^{1,\gamma}(\Omega)\), and test \eqref{eq:varIneq} with \((\mu,k)=(m,h)\) and \(v_{\rm test}=u+v+c\). Then
\[
0 \leq \left\langle A \begin{bmatrix} m \\ u \\ h \end{bmatrix}, \begin{bmatrix} 0 \\ v+c \\ 0 \end{bmatrix} \right\rangle
= \left\langle A \begin{bmatrix} m \\ u \\ h \end{bmatrix}, \begin{bmatrix} 0 \\ v \\ 0 \end{bmatrix} \right\rangle + c \left(-\int_{\Gamma_N} j\,\ds + \int_{\Gamma_D} h\,\ds \right).
\]
Dividing by \(|c|\) and letting \(|c| \to \infty\), we conclude that
\[
0 \leq \pm \left(-\int_{\Gamma_N} j \ds + \int_{\Gamma_D} h \ds\right),
\]
which is only possible if
\begin{equation}\label{eq:compatModDomSepA--1}
\int_{\Gamma_D} h\,\ds = \int_{\Gamma_N} j\,\ds.
\end{equation}

This compatibility condition aligns with the optimization problem \eqref{eq:DerSepA3}. It is also consistent with Lemma~\ref{lem:PropertiesOfcQ}, which imposes that the image of the modified monotone operator belongs to the set of annihilators of constant functions.
In other words, solutions to the variational inequality must satisfy the compatibility condition 
\eqref{eq:compatModDomSepA--1}.
Accordingly, a natural domain to look for a solution is 
\begin{equation} \label{eq:defcYsepA}
    \cY := \Big\{ (m, \bar{u}, h) \in L^{\beta+1}(\Omega)\times \cQ \times L^{\gamma'}(\Gamma_D) : m \geq 0,\; h \geq 0,\; \int_{\Gamma_D} h \ds=\int_{\Gamma_N} j \ds \Big\}.
\end{equation}
We note that \( \cY \) is convex (since $\cQ$ is a linear space and the constraints are affine/positivity).
Since this modification changes the domain's structure, the operator must be redefined accordingly; we address this in the next subsection.

\subsection{The modified operator} \label{subsec:ModSepOperator}

Defining an operator on a quotient space typically requires a set-valued formulation. A key observation is that the image of the operator lies in the annihilator of constant functions (Lemma~\ref{lem:PropertiesOfcQ}). By combining this property with the solution \( (m_u, h_u) \) from Subsection~\ref{subsec:SolvingFormAndh}, we construct a
single-valued operator that circumvents these difficulties.

We begin our analysis by selecting a canonical representative for each equivalence class \( \bar{u} \in \cQ \).
The choice of representative is governed by the following lemma.

\begin{lemma} \label{lem:choiceOfrepForbaru}
Suppose that Assumptions~\ref{assumset:D} hold, and let \( u \in W^{1,\gamma}(\Omega) \). Then, there exists a unique \( \kappa \in \mathbb{R} \) such that
\begin{equation}
\label{defk}
    \int_{\Gamma_D}  (u+\kappa)_+^{\gamma-1}  \ds= \int_{\Gamma_N} j\ds.
\end{equation}
\end{lemma}

\begin{proof} First, note that the following function is well-defined 
    \[
        f( \varsigma ) :=  \int_{\Gamma_D}  (u+ \varsigma )_+^{\gamma-1}\ds, 
    \]
because $u \in L^\gamma(\Gamma_D)$ due to the trace theorem. Moreover, $f$ is (strictly) increasing since for any \( \kappa_1 < \kappa_2 \) such that \( f(\kappa_1) > 0 \), we have
    \[
        f(\kappa_1) = \int_{\Gamma_D \cap \{ u+\kappa_1>0\}} (u+\kappa_1)^{\gamma-1}\ds
        < \int_{\Gamma_D \cap \{ u+\kappa_2>0\}} (u+\kappa_2)^{\gamma-1}\ds = f(\kappa_2).
    \]
Therefore, 
there is at most one value 
\( \kappa \) for which \eqref{defk} holds. 

To establish the existence of \( \kappa \), we first note that Assumption~\ref{assume:D.2} implies \( \int_{\Gamma_N} j \ds > 0 \). Thus, it suffices to show that
    \(
    (0, \infty) \subseteq f(\R) .
    \)
By the dominated convergence theorem,
\[
 \lim_{\varsigma  \to -\infty}f( \varsigma ) = 0.
\]
Furthermore,    
since $|\Gamma_D| > 0$ and $u \in L^\gamma(\Gamma_D)$, 
Chebyshev's
inequality implies that 
 $\mathcal{H}^{d-1}(\{|u| > \varsigma/2\}) \to 0$, so $f(\varsigma) \geq (\varsigma/2)^{\gamma-1} \mathcal{H}^{d-1}(\{|u| \leq \varsigma/2\}) \to \infty$.
Accordingly,
    we have
    \[
    \lim_{\varsigma  \to \infty} f( \varsigma ) = +\infty.
    \]
Moreover, for any $M>0$ and all $\varsigma \in [-M,M]$, we have 
    \[
    (u+ \varsigma)_+^{\gamma-1} \leq (|u|+ M)^{\gamma-1} \in  L^{1}(\Gamma_D),
    \]
    since $u \in L^{\gamma}(\Gamma_D)$ by the trace theorem. Therefore, $f$ is continuous by the dominated convergence theorem.
     Consequently, by the intermediate value theorem, \( f( \varsigma ) \) attains all values in \( (0,\infty) \); hence, there exists a solution $\kappa$ to \eqref{defk}.
\end{proof}

This lemma allows us to assign a unique representative \( u^* = u+\kappa \) to each \( \bar{u} \in \cQ \), where \( \kappa \) is the unique constant satisfying
\begin{equation} \label{eq:SepCompatibilityOfhu*}
        \int_{\Gamma_D}  h_{u^*}\ds = \int_{\Gamma_D}  (u+\kappa)_+^{\gamma-1}\ds  = \int_{\Gamma_N} j\ds.
\end{equation}
Lastly, we note that by Corollary~\ref{cor:FormulaFor_m_and_h}, the density \( m_u \) depends only on \( Du \). Since \( u^* = u + \kappa \) differs from \( u \) only by a constant, we have \( Du^* = Du \), and consequently \( m_{u^*} = m_u \). Thus, the modification of the domain affects only the boundary flux \( h_{u^*} \), while the interior transport density remains invariant.
This justifies defining the modified operator \( \bar{B}: \cQ \to \cQ' \) such that for any \( \bar{u}, \bar{v} \in \cQ \),
\begin{equation} \label{eq:defOfsepbarB}
    \langle \bar{B}\bar{u}, \bar{v} \rangle := \left\langle A \begin{bmatrix} m_{u}\\ u^* \\ h_{u^*} \end{bmatrix} \, ,\,  \begin{bmatrix} 0 \\ v \\ 0 \end{bmatrix} \right\rangle,
\end{equation}
where \( v \) is an arbitrary representative of \( \bar{v} \), and \( (m_u, u^*, h_{u^*}) \) is the compatible triplet derived from \( u \).

The operator \( \bar{B} \) is well-defined. The compatible triplet depends only on the class \( \bar{u} \).
Note that $m_u=m_{u^*}$ so either can be used in the definition of $\bar B$.
The compatibility condition \eqref{eq:SepCompatibilityOfhu*} guarantees that the pairing is independent of the representative \( v \). Indeed, for any constant \( c \in \R \),
\[
\left\langle A \begin{bmatrix} m_{u}\\ u^* \\ h_{u^*} \end{bmatrix} \, ,\,  \begin{bmatrix} 0 \\ c \\ 0 \end{bmatrix} \right\rangle
= c \left( -\int_{\Gamma_N} j \ds + \int_{\Gamma_D} h_{u^*} \ds \right) = 0.
\]

\subsection{The key properties of the modified operator}

In this subsection, we verify the monotonicity, hemicontinuity, and coercivity properties of \(\bar{B} \). We begin with the following proposition on the monotonicity of \( \bar{B} \).
\begin{proposition} \label{prop:monotonicityofSepbarB}
Suppose that Assumptions~\ref{assumset:D} and \ref{assumset:S} hold. Then, the operator $\bar{B}$ defined by \eqref{eq:defOfsepbarB} is monotone.
\end{proposition}
\begin{proof} 
    Let \(\bar{u}, \bar{v} \in \cQ\) be arbitrary. Let \( u \) and \( v \) denote arbitrary representatives, and let \(u^*\) and \( v^* \) denote their compatible representatives according to \eqref{eq:SepCompatibilityOfhu*}. 
    By the definition of \( \bar{B} \), we have
    \begin{align*}
        \langle \bar{B}\bar{u}-\bar{B}\bar{v}, \bar{u}-\bar{v} \rangle  
        &= \int_\Omega (m_{u} D_pH(x,Du) - m_{v}D_pH(x,Dv)) \cdot D(u-v) \dx \\
        &\phantom{={}} + \int_{\Gamma_D} (h_{u^*}-h_{v^*}) (u-v) \ds.
    \end{align*}
    For the first term, recall that $m_u = g^{-1}(H(x,Du))$. Let $G$ be the primitive of $g^{-1}$ such that $G(0)=0$. Since $g^{-1}$ is non-decreasing, $G$ is convex and non-decreasing. Since $H$ is convex, the composition $\Phi(p) := G(H(x,p))$ is convex. 
    Observing that $m_u D_pH(x,Du) = D_p\Phi(Du)$, the monotonicity of the gradient of the convex function $\Phi$ implies that the first integral is non-negative.
    For the boundary term, let \( c = (u-u^*)-(v-v^*) \). This is a constant. By the compatibility condition \eqref{eq:SepCompatibilityOfhu*}, we have \(\int_{\Gamma_D}(h_{u^*}-h_{v^*})\ds=0\), so
    \[
        \int_{\Gamma_D} (h_{u^*}-h_{v^*}) (u-v) \ds 
        = \int_{\Gamma_D} (h_{u^*}-h_{v^*}) (u^*-v^*) \ds + c \int_{\Gamma_D} (h_{u^*} - h_{v^*}) \ds
        = \int_{\Gamma_D} (h_{u^*}-h_{v^*}) (u^*-v^*) \ds.
    \]
    Since \( h_{w} = w_+^{\gamma-1} \) is non-decreasing in \( w \), the remaining integral is non-negative.
\end{proof}

The proof of hemicontinuity of the operator \( \bar{B} \) requires us to establish the following lemma on the continuity of the compatibility constant \( \kappa \), arising from Lemma~\ref{lem:choiceOfrepForbaru}.

\begin{lemma} \label{lem:HemicontoOfu*}
Suppose that Assumptions~\ref{assumset:D} hold.
Let \( \bar{u}_0, \bar{u}_1 \in \cQ \) be arbitrary, fix representatives \(u_0\in \bar u_0\) and \(u_1\in \bar u_1\), and set \(u_\th := (1-\th)u_0 + \th u_1\). Define \(\bar{u}_\th = (1-\th)\bar{u}_0 + \th \bar{u}_1\) for every \( \th \in [0,1] \).
Furthermore, let
    \[
        u^*_\th \quad \text{ and } \quad \kappa_{\th} := u_{\th}^* - u_{\th}
    \]
    respectively denote the compatible representative and the corresponding compatibility constant arising from Lemma~\ref{lem:choiceOfrepForbaru}.
    Then,
    \[
        \th \mapsto u_{\th}^*  \quad \text{and} \quad  \th \mapsto \kappa_{\th}
    \]
    are continuous for all \( \th \in [0,1] \).
\end{lemma}
\begin{proof}
Fix \(J:=\int_{\Gamma_N} j\,\ds>0\). Define
\[
\Phi(\varsigma,\th):=\int_{\Gamma_D}(u_\th+\varsigma)_+^{\gamma-1}\,\ds,
\qquad (\varsigma,\th)\in\R\times[0,1].
\]
For each fixed \(\varsigma\), the map \(\th\mapsto \Phi(\varsigma,\th)\) is continuous by the dominated convergence theorem, since \(u_\th(x)=(1-\th)u_0(x)+\th u_1(x)\) converges pointwise and
\((u_\th+\varsigma)_+^{\gamma-1}\le C\bigl(|u_0|+|u_1|+|\varsigma|\bigr)^{\gamma-1}\in L^1(\Gamma_D)\).
For each fixed \(\th\), the map \(\varsigma\mapsto \Phi(\varsigma,\th)\) is continuous and nondecreasing, with
\(\lim_{\varsigma\to-\infty}\Phi(\varsigma,\th)=0\) and \(\lim_{\varsigma\to+\infty}\Phi(\varsigma,\th)=+\infty\); hence, by Lemma~\ref{lem:choiceOfrepForbaru}, there exists a unique \(\kappa_\th\in\R\) such that \(\Phi(\kappa_\th,\th)=J\).

To prove continuity of \(\th\mapsto \kappa_\th\), fix \(\th_0\in[0,1]\).
Choose \(M>0\) such that \(\Phi(-M,\th_0)<J<\Phi(M,\th_0)\). By continuity of \(\th\mapsto\Phi(\pm M,\th)\), the same inequalities hold for all \(\th\) sufficiently close to \(\th_0\); hence \(\kappa_\th\in[-M,M]\) locally.
If \(\th_n\to\th_0\), extract a subsequence (not relabeled) with \(\kappa_{\th_n}\to \kappa_*\in[-M,M]\). Then, by continuity of \(\Phi\),
\[
J=\Phi(\kappa_{\th_n},\th_n)\to \Phi(\kappa_*,\th_0).
\]
By uniqueness at \(\th_0\), \(\kappa_*=\kappa_{\th_0}\). Therefore \(\kappa_{\th_n}\to\kappa_{\th_0}\), and \(\th\mapsto \kappa_\th\) is continuous. Finally, \(u_\th^*=u_\th+\kappa_\th\) is continuous as a sum of continuous maps.
\end{proof}

The following result is a direct corollary of Lemma~\ref{lem:HemicontoOfu*}.
\begin{corollary}\label{cor:corOfHemicontoOfu*}
Suppose that Assumptions~\ref{assumset:D} hold. Furthermore, let \(\bar{u}_0, \bar{u}_1 \in \cQ\) be arbitrary, and let \(\kappa_{\th}\) denote the compatibility constant defined in Lemma~\ref{lem:HemicontoOfu*}. Then, there exists a constant \( C > 0 \), independent of \(\th \in [0,1]\), such that
    \[
        |\kappa_{\th}| \leq C.
    \]
\end{corollary}
\begin{proof}
    Due to Lemma~\ref{lem:HemicontoOfu*}, \( \th \mapsto \kappa_{\th} \) is a continuous map on a compact set; hence, it is bounded.
\end{proof}

With the continuity of the compatibility constant established, we now prove the hemicontinuity of the operator \(\bar{B}\).
\begin{proposition} \label{prop:hemicontofSepbarB}
    Suppose that Assumptions~\ref{assumset:D} and \ref{assumset:S} hold. Then, the operator \(\bar{B}\) is hemicontinuous. More specifically, let \( \bar{u}_0, \bar{u}_1, \bar{v} \in \cQ \) be fixed. Then, the map
    \[
        \th \mapsto \left\langle \bar{B}[(1-\th)\bar{u}_0 + \th \bar{u}_1]  \, ,\,  \bar{v}  \right\rangle
    \]
    is continuous for all \( \th \in [0,1] \).
\end{proposition}
\begin{proof}
    Let \(\bar{u}_0, \bar{u}_1 \in \cQ\) and define
    \[
        \bar{u}_\th = (1-\th)\bar{u}_0 + \th \bar{u}_1 \quad \text{for } \th \in [0,1],
    \]
    and let \( u^*_\th \) and \( \kappa_{\th} := u_{\th}^* - u_{\th} \) be defined as before. 
    Fix \(\bar v\in\cQ\), and let \(v\in W^{1,\gamma}(\Omega)\) be an arbitrary representative of \(\bar v\).
Define the map
    \[
    \varpi(\th):=  \left\langle \bar{B}\bar{u}_\th  \, ,\,  \bar{v}  \right\rangle .
    \]
    Due to Corollary~\ref{cor:FormulaFor_m_and_h} and Lemma~\ref{lem:boundOng^-1DpH1}, we have
    \[
        \begin{array}{>{\displaystyle}r>{\displaystyle}r>{\displaystyle}l}
            m_{u_{\th}} \left|D_pH(x,Du_{\th}(x))\right|
            = & G'(H(x,Du_{\th}(x))) \left|D_pH(x,Du_{\th})\right|

            \\[.25em]

        & \leq C \left( |Du_1|^{\gamma-1} +  |Du_0|^{\gamma-1}  +1 \right)
        &\in L^{\gamma'}(\Omega)
        \end{array}
    \]
    for all \( \th \in [0,1] \). In addition, Corollary~\ref{cor:corOfHemicontoOfu*} and some basic inequalities give us that
    \[
        \begin{split}
        h_{u_{\th}^*} &=  (u_{\th}+\kappa_\th)_+^{\gamma-1} \leq C (  (u_0)_+^{\gamma-1} + (u_1)_+^{\gamma-1} + 1)
        \end{split}
    \]
    for some \( C > 1 \), independent of \( \th \). 
    Expanding the definition of \( \bar{B} \), we have
    \[
    \varpi(\th)=\int_\Omega m_{u_{\th}} D_pH(x,Du_{\th}) \cdot Dv \dx -\int_{\Gamma_N}  j v \ds +\int_{\Gamma_D}  h_{u_{\th}^*} v \ds.
    \] 
Pointwise continuity (a.e.\ in \(\Omega\) and on \(\Gamma_D\)) in \(\th\) follows from the continuity of \(p\mapsto H(x,p)\), \(p\mapsto D_pH(x,p)\), and of \(g^{-1}\), together with Lemma~\ref{lem:HemicontoOfu*}.    
The uniform bounds above provide integrable majorants, so the dominated convergence theorem yields the continuity of \(\varpi\) on \([0,1]\).
\end{proof}

The last result in this subsection is the coercivity of the operator \( \bar{B} \), which we establish in the following proposition.
\begin{proposition} \label{prop:coercivityofSepbarB}
    Suppose that Assumptions~\ref{assumset:D} and \ref{assumset:S} hold. Then,
    \begin{equation*}
        \frac{1}{\|\bar{u}\|_{\cQ}}\langle \bar{B}\bar{u}-\bar{B}\bar{0}, \bar{u} \rangle \to \infty \quad \text{ as } \quad \|\bar{u}\|_{\cQ} \to \infty.
    \end{equation*}
\end{proposition}
\begin{proof}
 	First, note that
 	\[
 	 	m_{0} = g^{-1}(H(x,0)), \quad 0^* = \left( \frac{1}{|\Gamma_D|} \int_{\Gamma_N} j \ds\right)^{\gamma'-1}, \enskip  \text{ and } \enskip h_{0^*} = \frac{1}{|\Gamma_D|} \int_{\Gamma_N} j\ds.
 	\]
 Furthermore,
\begin{align*}
 	\langle \bar{B}\bar{u}-\bar{B}\bar{0}, \bar{u} \rangle
 	&= \int_\Omega \left( m_{u} D_pH(x,Du) - m_{0} D_pH(x,0) \right) \cdot Du \dx \\
 	&\phantom{={}} + \int_{\Gamma_D} (h_{u^*}-h_{0^*}) u \ds.
\end{align*}
Due to Corollary~\ref{cor:boundOng^-1(H)1} and Assumption~\ref{assume:S.5}, we have the lower bounds
\[
m_{u}=g^{-1}(H(x,Du)) \geq C^{-1}|Du|^{\alpha/\beta}-C,
\qquad
D_pH(x,Du)\cdot Du \geq C^{-1}|Du|^{\alpha}-C.
\]
Consequently,
\[
m_{u}D_pH(x,Du)\cdot Du \geq \frac{1}{C}|Du|^{\gamma}-C,
\qquad
\gamma=\alpha+\frac{\alpha}{\beta}.
\]
Moreover, since \(m_{0}D_pH(x,0)\in L^{\gamma'}(\Omega;\R^d)\), H\"older and Young yield
\[
\left|\int_\Omega m_{0}D_pH(x,0)\cdot Du\,\dx\right|
\le \|m_{0}D_pH(x,0)\|_{L^{\gamma'}(\Omega)}\|Du\|_{L^\gamma(\Omega)}
\le \frac{1}{2C}\|Du\|_{L^\gamma(\Omega)}^\gamma + C.
\]
Since $\int_{\Gamma_D}h_{u^*}\,ds=\int_{\Gamma_N}j\,ds=\int_{\Gamma_D}h_{0^*}\,ds$,
we have $\int_{\Gamma_D}(h_{u^*}-h_{0^*})\,ds=0$. Hence
\[
\int_{\Gamma_D}(h_{u^*}-h_{0^*})u\,ds
=\int_{\Gamma_D}(h_{u^*}-h_{0^*})(u^*-0^*)\,ds.
\]
Because $h_w=w_+^{\gamma-1}$ and $s\mapsto s_+^{\gamma-1}$ is monotone,
the integrand is nonnegative a.e., and therefore the integral is $\ge0$.

Combining these estimates and using \(\|\bar u\|_{\cQ}=\|Du\|_{L^\gamma(\Omega)}\), we obtain
\[
\langle \bar{B}\bar{u}-\bar{B}\bar{0}, \bar{u} \rangle
\ge \frac{1}{C}\|\bar{u}\|_{\cQ}^{\gamma}-C.
\]
 	We conclude the proof by dividing both sides by \( \|\bar{u}\|_{\cQ} \), and taking the limit \( \|\bar{u}\|_{\cQ} \to \infty \).
\end{proof}

\subsection{The existence of solutions for the penalized problem} \label{subsec:existenceOfSolSepBandSepA}

In this subsection, we establish the existence of solutions for the modified operator \( \bar{B} \) and subsequently prove the existence of solutions for the original operator \( A \).

\begin{theorem}\label{thm:ExistenceOfSolutionsForSepbarB}
Suppose that Assumptions~\ref{assumset:D} and \ref{assumset:S} hold. Then, there exists a $\bar{u} \in \cQ$ such that
    \begin{equation*}
        \langle \bar{B}\bar{u}, \bar{v}-\bar{u} \rangle \geq 0
    \end{equation*}
    for all $\bar{v} \in \cQ$.
\end{theorem}
\begin{proof}
    Due to Propositions~\ref{prop:monotonicityofSepbarB},~\ref{prop:hemicontofSepbarB}, and~\ref{prop:coercivityofSepbarB}, the theorem is a direct consequence of 
    Lemma~\ref{lem:UnboundedMintyExistence}. 
    Moreover, since the domain \( \cQ \) is a vector space, we can choose \( \bar{v} = \bar{u} \pm \bar{w} \) to conclude that
\begin{equation} \label{eq:BarB_equality}
\langle \bar{B}\bar{u}, \bar{w} \rangle = 0 \quad \text{for all } \bar{w} \in \cQ.
    \end{equation}
\end{proof}

We now establish the existence of solutions for the original operator $A$.
\begin{theorem}\label{thm:ExistenceOfSolutionsForSepA}
    Suppose that Assumptions~\ref{assumset:D} and \ref{assumset:S} hold. Then, there exists a triplet $(m,u,h) \in \cX^+$ such that
    \begin{equation*}
	\left\langle A \begin{bmatrix} m \\ u  \\ h \end{bmatrix}\, , \, \begin{bmatrix} \mu-m \\ v-u \\ k-h\end{bmatrix}  \right \rangle  \geq 0
    \end{equation*}
    for all $(\mu,v,k) \in \cX^+$.
\end{theorem}

\begin{proof}
    We construct the solution based on Theorem~\ref{thm:Existence_of_m_and_h}, Theorem~\ref{thm:ExistenceOfSolutionsForSepbarB}, and Corollary~\ref{cor:FormulaFor_m_and_h}.

    Let $\bar{u}$ be the solution for $\bar{B}$, arising from Theorem~\ref{thm:ExistenceOfSolutionsForSepbarB}, and let $u \in \bar{u}$ be the member satisfying
    \begin{equation*}
        \int_{\Gamma_D} u_+^{\gamma-1}\ds = \int_{\Gamma_N} j \ds,
    \end{equation*}
which exists due to Lemma~\ref{lem:choiceOfrepForbaru}. Let $m = m_{u}$ and $h = h_{u}$ be the functions arising from Corollary~\ref{cor:FormulaFor_m_and_h}. Then, for any $(\mu,v,k) \in \cX^+$, we define $\bar{v}$ as the equivalence class of $v$ in $\mathcal{Q}$. Using \eqref{eq:BarB_equality} and the definition of $A_u$, we obtain
    \begin{equation*}
    \begin{split}
    \left\langle A \begin{bmatrix} m \\ u  \\ h \end{bmatrix}\, , \, \begin{bmatrix} \mu-m \\ v-u \\ k-h\end{bmatrix}  \right \rangle
    &=  \left\langle A \begin{bmatrix} m_u \\ u  \\ h_u \end{bmatrix}\, , \, \begin{bmatrix} \mu-m \\ 0 \\ k-h\end{bmatrix}  \right \rangle
    + \left\langle A \begin{bmatrix} m_u \\ u  \\ h_u \end{bmatrix}\, , \, \begin{bmatrix} 0\\ v-u \\ 0\end{bmatrix}  \right \rangle \\[1.45em]
    &= \left\langle A_u \begin{bmatrix} m_u \\ h_u \end{bmatrix}\, , \, \begin{bmatrix} \mu-m  \\ k-h\end{bmatrix}  \right \rangle
    + \left\langle \bar{B} \bar{u}\, , \, \bar{v} -\bar{u} \right \rangle.
    \end{split}
    \end{equation*}
    The first term is non-negative by Theorem~\ref{thm:Existence_of_m_and_h}, and the second term is zero by \eqref{eq:BarB_equality}. Thus, the sum is non-negative.
\end{proof}

The general case of $A_\epsilon$ for $\epsilon >0$ is given by the following theorem.
\begin{theorem}[\bfseries Existence of solutions for $A_{\epsilon}$] \label{thm:ExistenceOfSolutionsForSepA_ep}
    Suppose that Assumptions~\ref{assumset:D} and \ref{assumset:S} hold. Then, there exists a triplet $(m,u,h) \in \cX^+$ such that
    \begin{equation*}
	\left\langle A_{\epsilon} \begin{bmatrix} m \\ u  \\ h \end{bmatrix}\, , \, \begin{bmatrix} \mu-m \\ v-u \\ k-h\end{bmatrix}  \right \rangle  \geq 0.
    \end{equation*}
    for all $(\mu,v,k) \in \cX^+$.
\end{theorem}
\begin{proof}
The proof follows the same lines as that of Theorem~\ref{thm:ExistenceOfSolutionsForSepA}, noting that 
\begin{equation}\label{eq:h_uFormulaForA_ep}
     h_{u} = \frac{u_+^{\gamma-1}}{\epsilon^{\gamma}}.
\end{equation}
Accordingly, the compatible representative is chosen so that \(\int_{\Gamma_D}(u+\kappa)_+^{\gamma-1}\,\ds=\epsilon^\gamma\int_{\Gamma_N}j\,\ds\).
\end{proof}

%****************

\section{Passage to the Limit and Recovery of the MFG System} \label{sec:backToMFG}

In this section, we establish the existence of solutions to MFG~System~\ref{mfg:2} by passing to the limit \( \epsilon \to 0^+ \) in the family of solutions to the penalized variational inequalities associated with \(A_{\epsilon}\). This requires uniform a priori estimates as \( \epsilon \to 0^+ \). In what follows, \( (m_{\epsilon}, u_{\epsilon}, h_{\epsilon}) \) denotes a solution (whose existence is ensured by Theorem~\ref{thm:ExistenceOfSolutionsForSepA_ep}) to the penalized variational inequality
\begin{equation}\label{eq:varIneq_epLast}
	\left\langle A_{\epsilon} \begin{bmatrix} m_{\epsilon} \\ u_{\epsilon}  \\ h_{\epsilon} \end{bmatrix}, \begin{bmatrix} \mu - m_{\epsilon} \\ v- u_{\epsilon} \\ k-h_{\epsilon}\end{bmatrix}  \right \rangle  \geq 0 \qquad \text{for all } (\mu, v, k)  \in \cX^+.
\end{equation}
As in the proof of Corollary~\ref{cor:FormulaFor_m_and_h}, 
testing \eqref{eq:varIneq_epLast} with $(\mu,v,k)=(\mu,u_\epsilon,h_\epsilon)$ and varying $\mu\ge0$, 
and with $(\mu,v,k)=(m_\epsilon,u_\epsilon,k)$ and varying $k\ge0$, yields
\[
m_\epsilon = g^{-1}(H(x,Du_\epsilon)) \quad \text{a.e.\ in }\Omega,
\qquad
h_\epsilon = \epsilon^{-\gamma}(u_\epsilon)_+^{\gamma-1} \quad \text{a.e.\ on }\Gamma_D.
\]

\begin{remark} \label{rmk:PairingAgainstSolutions}
Testing \eqref{eq:varIneq_epLast} with \( ( \mu , v , k ) = (2m_{\epsilon}, 2u_{\epsilon}, 2h_{\epsilon}) \) and with \( ( \mu , v , k ) = (0,0,0 ) \), we obtain
    \[
       \left\langle A_{\epsilon} \begin{bmatrix} m_{\epsilon} \\ u_{\epsilon}  \\ h_{\epsilon} \end{bmatrix}, \begin{bmatrix}  m_{\epsilon} \\ u_{\epsilon} \\ h_{\epsilon}\end{bmatrix}  \right \rangle   = 0.
    \]
\end{remark}

We recall two Poincar\'e-type inequalities with boundary terms. First, recall that \(\mathcal{H}^{d-1}(\Gamma_D)>0\). 
Then for any \(u\in W^{1,\gamma}(\Omega)\), there exists a constant \(C>0\), depending only on \(\Omega\), \(\Gamma_D\), and \(\gamma\), such that
\begin{equation}
\label{poincare}
    \|u\|_{W^{1,\gamma}(\Omega)} \le C \left( \|Du\|_{L^\gamma(\Omega)} + \|u\|_{L^\gamma(\Gamma_D)} \right).
\end{equation}
The second version of the Poincar\'e inequality anchored by the inflow is stated in the following lemma. 

\begin{lemma}\label{lem:poincare_j}
Suppose Assumption~\ref{assume:D.2} holds and that 
\(\mathcal{H}^{d-1}(\Gamma_N)>0\). Let \(J_{\mathrm{in}}:=\int_{\Gamma_N} j\,\ds\).
Then there exists \(C>0\), depending only on \(\Omega\), \(\gamma\), \(J_{\mathrm{in}}\), and \(\|j\|_{L^{\gamma'}(\Gamma_N)}\), such that for every \(u\in W^{1,\gamma}(\Omega)\),
\[
\|u\|_{W^{1,\gamma}(\Omega)}
\le C\Bigl(\|Du\|_{L^\gamma(\Omega)} + \Big|\int_{\Gamma_N} j\,u\ds\Big|\Bigr).
\]
\end{lemma}
\begin{proof}
By Assumption~\ref{assume:D.2} and \(\mathcal{H}^{d-1}(\Gamma_N)>0\), since \(j\not\equiv 0\) and \(j\ge 0\), we have \(J_{\mathrm{in}}>0\).
To prove the inequality in the lemma, we argue by contradiction. Suppose the claim fails. Then for each \(n\in\mathbb{N}\) there exists \(u_n\in W^{1,\gamma}(\Omega)\) such that
\[
\|u_n\|_{W^{1,\gamma}(\Omega)}=1
\quad\text{and}\quad
\|Du_n\|_{L^\gamma(\Omega)} + \left|\int_{\Gamma_N} j\,u_n\ds\right| \le \frac1n.
\]
Let \(c_n = \frac{1}{|\Omega|}\int_{\Omega} u_n\dx\). By the standard Poincar\'e inequality on \(\Omega\),
\[
\|u_n-c_n\|_{L^\gamma(\Omega)} \le C\,\|Du_n\|_{L^\gamma(\Omega)} \xrightarrow[n\to\infty]{}0,
\]
and, hence, \(u_n-c_n\to 0\) in \(W^{1,\gamma}(\Omega)\) (since also \(\|D(u_n-c_n)\|_{L^\gamma}=\|Du_n\|_{L^\gamma}\to 0\)).
By the trace theorem, it follows that
\[
\|u_n-c_n\|_{L^\gamma(\Gamma_N)} \to 0.
\]
Recalling that \(J>0\), we have
\[
|c_n|\,J_{\mathrm{in}}
= \left|\int_{\Gamma_N} j\,c_n\ds\right|
\le \left|\int_{\Gamma_N} j\,u_n\ds\right|
+ \left|\int_{\Gamma_N} j\,(u_n-c_n)\ds\right|
\le \left|\int_{\Gamma_N} j\,u_n\ds\right| + \|j\|_{L^{\gamma'}(\Gamma_N)}\|u_n-c_n\|_{L^\gamma(\Gamma_N)} \to 0,
\]
so \(c_n\to 0\). Therefore,
\[
\|u_n\|_{W^{1,\gamma}(\Omega)}
\le \|u_n-c_n\|_{W^{1,\gamma}(\Omega)} + \|c_n\|_{W^{1,\gamma}(\Omega)}
= \|u_n-c_n\|_{W^{1,\gamma}(\Omega)} + \|c_n\|_{L^\gamma(\Omega)} \to 0,
\]
which contradicts \(\|u_n\|_{W^{1,\gamma}(\Omega)}=1\). This proves the lemma.
\end{proof}

Our first result is a general estimate that connects the integral information on the Neumann boundary \(\Gamma_N \) to the integral information on the Dirichlet boundary \( \Gamma_D \) and the interior of \( \Omega\).
\begin{lemma} \label{lem:upBoundjuPSIneq1}
	Suppose Assumptions~\ref{assumset:D} and \ref{assumset:S} hold. Let \( \dd >0\) and let \(u \in W^{1,\gamma}(\Omega) \).  Then, there is a constant \( C_{\dd} > 1\), independent of \( u \), such that
    \[
		\int_{\Gamma_N} j  u \ds \leq \dd \, \|Du\|_{L^\gamma(\Omega)}^{\gamma} + \|u_+\|^{\gamma}_{L^\gamma(\Gamma_D)} + C_{\dd}.
    \]
\end{lemma}
\begin{proof}
Since \(j\ge 0\), we have
\[
\int_{\Gamma_N} j u \ds \le \int_{\Gamma_N} j u_+ \ds
\le \|j\|_{L^{\gamma'}(\Gamma_N)} \,\|u_+\|_{L^\gamma(\Gamma_N)}
\]
by H\"older's inequality.
By the Trace Theorem and the Poincar\'e inequality \eqref{poincare} applied to \(u_+\in W^{1,\gamma}(\Omega)\),
\begin{equation}\label{eq:lem:upBoundjuPSIneq1}
\|u_+\|_{L^\gamma(\Gamma_N)} \le C\|u_+\|_{W^{1,\gamma}(\Omega)}
\le C\Bigl(\|Du_+\|_{L^\gamma(\Omega)} + \|u_+\|_{L^\gamma(\Gamma_D)}\Bigr).
\end{equation}
Using \(\|Du_+\|_{L^\gamma(\Omega)} \le \|Du\|_{L^\gamma(\Omega)}\) and Young's inequality, for any \(\dd>0\) we obtain
\[
\int_{\Gamma_N} j u \ds
\le \dd\,\|Du\|_{L^\gamma(\Omega)}^\gamma
+ \|u_+\|_{L^\gamma(\Gamma_D)}^\gamma
+ C_\dd,
\]
for a constant \(C_\dd>1\) independent of \(u\).
\end{proof}

Next, we establish the following estimate, which gives a bound on the growth of the solution as \( \epsilon \to 0^+\).

\begin{proposition} \label{pro:upBoundjuPSIneq2}
	Under Assumptions~\ref{assumset:D} and \ref{assumset:S}, let \( 0 < \epsilon < \frac{1}{2}\). Let \( (m_{\epsilon}, u_{\epsilon}, h_{\epsilon}) \in \cX^+ \) be a solution to \eqref{eq:varIneq_epLast}.
    Then, there is a constant \(C >0\), independent of \( \epsilon \), such that
    \[
       \epsilon^{-\gamma} \int_{\Gamma_D} (u_\epsilon)_+^{\gamma}\ds + \|Du_\epsilon\|_{L^\gamma(\Omega)}^{\gamma} \leq C.
    \]
\end{proposition}
\begin{proof}
To simplify the notation, \( (m,u,h) = (m_{\epsilon}, u_{\epsilon}, h_{\epsilon}) \). By Remark~\ref{rmk:PairingAgainstSolutions} and the definition of the pairing \eqref{eq:A_rigorous_pairing}, we have
\begin{align}
\notag
0
&= \left\langle A_{\epsilon} \begin{bmatrix} m \\ u \\ h \end{bmatrix}, \begin{bmatrix} m \\ u \\ h \end{bmatrix} \right\rangle\\\label{eq:energy_identity_eps}
&= \int_\Omega m\bigl(g(m)-H(x,Du)\bigr)\dx + \int_\Omega m D_pH(x,Du)\cdot Du\dx - \int_{\Gamma_N} j u\ds + \epsilon^{\gamma'}\int_{\Gamma_D} h^{\gamma'}\ds.
\end{align}
Note that the boundary terms $\int_{\Gamma_D}hu\,\ds$ and $\int_{\Gamma_D}(-u)h\,\ds$ cancel in the pairing,
leaving only $\epsilon^{\gamma'}\int_{\Gamma_D}h^{\gamma'}\,\ds$.
By Corollary~\ref{cor:FormulaFor_m_and_h} (and Remark~\ref{rmk:ForCor:FormulaFor_m_and_h} for the \(\epsilon\)-dependence), we have \(m(g(m)-H(x,Du))=0\) a.e. in \(\Omega\) and \(\epsilon^{\gamma'} h^{\gamma'} = \epsilon^{-\gamma}u_+^\gamma\) a.e. on \(\Gamma_D\). Hence \eqref{eq:energy_identity_eps} reduces to
\begin{equation}\label{eq:energy_identity_eps2}
\int_\Omega m D_pH(x,Du)\cdot Du\dx + \epsilon^{-\gamma}\int_{\Gamma_D} u_+^{\gamma}\ds
= \int_{\Gamma_N} j u\ds.
\end{equation}
By Assumption~\ref{assume:S.5} (in particular \ref{assume:S.5.b}), Corollary~\ref{cor:boundOng^-1(H)1}, and Young's inequality, we obtain a lower bound of the form
\[
m D_pH(x,Du)\cdot Du \ge C^{-1}|Du|^{\gamma} - C.
\]
Integrating and inserting into \eqref{eq:energy_identity_eps2} gives
\[
\epsilon^{-\gamma}\int_{\Gamma_D} u_+^{\gamma}\ds + C^{-1}\|Du\|_{L^\gamma(\Omega)}^{\gamma} - C
\le \int_{\Gamma_N} j u\ds.
\]
Finally, apply Lemma~\ref{lem:upBoundjuPSIneq1} to the right-hand side of \eqref{eq:energy_identity_eps2} and choose \(\dd>0\) small enough to absorb the \(\dd\|Du\|_{L^\gamma(\Omega)}^\gamma\) term. Rearranging then gives
\[
(\epsilon^{-\gamma}-1)\int_{\Gamma_D} u_+^{\gamma}\ds + \|Du\|_{L^\gamma(\Omega)}^{\gamma} \leq C,
\]
for a constant \(C>0\) independent of \(\epsilon\). Since \(0<\epsilon<\tfrac12\),
the stated estimate follows (with a possibly larger constant \(C\)).
\end{proof}

\begin{lemma}\label{lem:hAnotherForm}
    Suppose Assumptions~\ref{assumset:S} and \ref{assumset:D} hold. Let \( (m_{\epsilon}, u_{\epsilon}, h_{\epsilon}) \in \cX^+ \) be a solution to \eqref{eq:varIneq_epLast}.
Then
\begin{equation}
\label{div0}
\Div\bigl(m_{\epsilon} D_pH(x,Du_{\epsilon})\bigr)=0
\end{equation}
in $\mathcal{D}'(\Omega)$, 
\(
h_{\epsilon} = - m_{\epsilon} D_pH(x,Du_{\epsilon}) \cdot \nu \quad \text{on } \Gamma_D,
\)
and $m_\epsilon D_pH(x,Du_\epsilon)\cdot\nu=j$ on $\Gamma_N$,
in the sense of $W^{-1+\frac1\gamma,\gamma'}(\Gamma)$.
\end{lemma}
\begin{proof}
Setting \((\mu,k)=(0,0)\) in \eqref{eq:varIneq_epLast} and using Remark~\ref{rmk:PairingAgainstSolutions}, we obtain for every \(v\in W^{1,\gamma}(\Omega)\),
\[
0 \le \int_\Omega m_{\epsilon} D_pH(x,Du_{\epsilon})\cdot Dv \dx - \int_{\Gamma_N} j v\ds + \int_{\Gamma_D} h_{\epsilon} v\ds.
\]
Replacing \(v\) by \(-v\) yields the reverse inequality, hence
\begin{equation}\label{eq:weakFluxIdentity_eps}
\int_\Omega m_{\epsilon} D_pH(x,Du_{\epsilon})\cdot Dv \dx
- \int_{\Gamma_N} j\, v\ds + \int_{\Gamma_D} h_{\epsilon}\, v\ds = 0
\qquad \forall v\in W^{1,\gamma}(\Omega),
\end{equation}
where boundary terms are understood via the trace of \(v\).

In particular, taking \(v\in C_c^\infty(\Omega)\) (or \(v\in W^{1,\gamma}_0(\Omega)\)) in \eqref{eq:weakFluxIdentity_eps} eliminates the boundary terms and yields
\[
\int_\Omega m_{\epsilon} D_pH(x,Du_{\epsilon})\cdot Dv \dx = 0
\qquad \forall v\in C_c^\infty(\Omega),
\]
so \eqref{div0} holds. 
Set \(F_{\epsilon}:=m_{\epsilon} D_pH(x,Du_{\epsilon})\).
By Proposition~\ref{prop:WellDefinedness}, \(F_{\epsilon}\in L^{\gamma'}(\Omega;\R^d)\), and the previous step gives \(\Div F_{\epsilon}=0\) in \(\mathcal{D}'(\Omega)\) (hence, \(\Div F_{\epsilon}\) is the zero element of \(L^{\gamma'}(\Omega)\)).
Therefore, Theorem~2.8 in \cite{Alharbi2026MFG} applies and yields a normal trace \(F_{\epsilon}\cdot\nu \in W^{-1+\frac1\gamma,\gamma'}(\Gamma)\) such that
\[
\int_\Omega F_{\epsilon}\cdot Dv \dx
= \big\langle F_{\epsilon}\cdot \nu,\; v|_{\Gamma}\big\rangle_{\Gamma}
\qquad \forall v\in W^{1,\gamma}(\Omega).
\]
Combining this identity with \eqref{eq:weakFluxIdentity_eps}, we obtain
\[
\big\langle F_{\epsilon}\cdot\nu + \Chi_{\Gamma_D}h_{\epsilon} - \Chi_{\Gamma_N}j,\; v|_{\Gamma}\big\rangle_{\Gamma}=0
\qquad \forall v\in W^{1,\gamma}(\Omega).
\]
Hence,
\[
F_{\epsilon}\cdot\nu + \Chi_{\Gamma_D}h_{\epsilon} - \Chi_{\Gamma_N}j = 0
\quad \text{in } W^{-1+\frac1\gamma,\gamma'}(\Gamma),
\]
and, in particular, \(F_{\epsilon}\cdot\nu = -h_{\epsilon}\) on \(\Gamma_D\), i.e.,
\(h_{\epsilon} = -m_{\epsilon} D_pH(x,Du_{\epsilon})\cdot\nu\) on \(\Gamma_D\).
\end{proof}

Finally, we combine the preceding results to establish the existence theorem for MFG System~\ref{mfg:2} from the introduction.
\begin{proof}[\textbf{Proof of Theorem~\ref{thm:ExistInSepMFG}}]
Let \(0<\epsilon<\frac12\), and let \((m_{\epsilon},u_{\epsilon},h_{\epsilon})\in\cX^+\) be a solution of \eqref{eq:varIneq_epLast}.

\paragraph{Step 1: Uniform bounds and weak compactness.}
From Proposition~\ref{pro:upBoundjuPSIneq2}, we have
\begin{equation}\label{eq:pfOf_ExistInSepMFG1}
\|Du_{\epsilon}\|_{L^\gamma(\Omega)}^{\gamma} + \epsilon^{-\gamma} \int_{\Gamma_D} (u_\epsilon)_+^{\gamma}\ds \leq C.
\end{equation}
In particular, \(\|Du_{\epsilon}\|_{L^\gamma(\Omega)} \le C\).
Moreover, \(\|(u_\epsilon)_+\|_{L^\gamma(\Gamma_D)}^\gamma \le C\epsilon^\gamma \to 0\).

Using Corollary~\ref{cor:boundOng^-1(H)1} and Corollary~\ref{cor:FormulaFor_m_and_h}, we have \(m_{\epsilon}=g^{-1}(H(x,Du_{\epsilon}))\) and therefore
\[
\|m_{\epsilon}\|_{L^{\beta+1}(\Omega)} \le C.
\]
Moreover, by Assumption~\ref{assume:S.4} and the same Young-inequality argument used in Proposition~\ref{prop:WellDefinedness} (transport-flux estimate), we have
\begin{equation}\label{eq:uniform_flux_bound}
\|m_{\epsilon}D_pH(x,Du_{\epsilon})\|_{L^{\gamma'}(\Omega;\R^d)}^{\gamma'}
\le C\Bigl(\|m_{\epsilon}\|_{L^{\beta+1}(\Omega)}^{\beta+1}+\|Du_{\epsilon}\|_{L^\gamma(\Omega)}^{\gamma}+1\Bigr).
\end{equation}
In particular, combining \eqref{eq:pfOf_ExistInSepMFG1} with \(\|m_\epsilon\|_{L^{\beta+1}(\Omega)}\le C\), we obtain
\[
\|m_{\epsilon}D_pH(x,Du_{\epsilon})\|_{L^{\gamma'}(\Omega;\R^d)} \le C.
\]
By Lemma~\ref{lem:hAnotherForm} and the normal-trace estimate (see \cite{Alharbi2026MFG}, Theorem~2.8),
\[
\|h_{\epsilon}\|_{W^{-1+\frac1\gamma,\gamma'}(\Gamma_D)}
\le C \|m_{\epsilon}D_pH(x,Du_{\epsilon})\|_{L^{\gamma'}(\Omega;\R^d)} \le C.
\]
Finally, by Lemma~\ref{lem:poincare_j} we have
\[
\|u_{\epsilon}\|_{W^{1,\gamma}(\Omega)}
\le C\Bigl(\|Du_{\epsilon}\|_{L^\gamma(\Omega)} + \Big|\int_{\Gamma_N} j\,u_{\epsilon}\ds\Big|\Bigr).
\]
Using \eqref{eq:energy_identity_eps2}, H\"older's inequality, and \eqref{eq:uniform_flux_bound}, we obtain
\begin{align*}
\left|\int_{\Gamma_N} j\,u_{\epsilon}\ds\right|
&= \left| \int_\Omega m_{\epsilon}D_pH(x,Du_{\epsilon})\cdot Du_{\epsilon}\dx
      + \epsilon^{-\gamma}\int_{\Gamma_D} (u_\epsilon)_+^\gamma\ds \right|\\
&\le \|m_{\epsilon}D_pH(x,Du_{\epsilon})\|_{L^{\gamma'}(\Omega;\R^d)}\|Du_{\epsilon}\|_{L^\gamma(\Omega)} + C
\le C.
\end{align*}
Therefore,
\[
\|u_{\epsilon}\|_{W^{1,\gamma}(\Omega)} \le C.
\]

Weak compactness in $W^{1,\gamma}(\Omega)$ implies that there exists $u\in W^{1,\gamma}(\Omega)$ and a subsequence such that 
\(u_\epsilon\rightharpoonup u\) in \(W^{1,\gamma}(\Omega)\).
By continuity of the trace operator, $u_\epsilon|_{\Gamma}\rightharpoonup u|_{\Gamma}$ in $W^{1-\frac1\gamma,\gamma}(\Gamma)$ (and hence in $L^\gamma(\Gamma)$), in particular on $\Gamma_N$ and $\Gamma_D$.
Since the functional \(w\mapsto \int_{\Gamma_D}(w_+)^\gamma\,\ds\) is convex and weakly lower semicontinuous on \(L^\gamma(\Gamma_D)\),
\[
\int_{\Gamma_D}(u_+)^\gamma\,\ds
\le \liminf_{\epsilon\to0^+}\int_{\Gamma_D}((u_\epsilon)_+)^\gamma\,\ds=0.
\]
Hence, \(u\le 0\) a.e.\ on \(\Gamma_D\).

To summarize, the preceding bounds imply there exists a subsequence (not relabeled) and \((m,u,h)\) such that
\[
m_{\epsilon}\rightharpoonup m \ \text{ in }L^{\beta+1}(\Omega),\qquad
u_{\epsilon}\rightharpoonup u \ \text{ in }W^{1,\gamma}(\Omega),\qquad
h_{\epsilon}\rightharpoonup h \ \text{ in }W^{-1+\frac1\gamma,\gamma'}(\Gamma_D).
\]
Since $m_\epsilon\ge0$ and $m_\epsilon\rightharpoonup m$ in $L^{\beta+1}(\Omega)$, we have $m\ge0$ a.e.\ in $\Omega$.
Similarly, 
since $h_\epsilon\ge0$ and $h_\epsilon\rightharpoonup h$ in $W^{-1+\frac1\gamma,\gamma'}(\Gamma_D)$,
we have $\langle h,\phi\rangle\ge0$ for all $\phi\ge0$ by weak convergence. Thus, \(h\ge 0\) in the sense of distributions  on 
\(\Gamma_D\).

\paragraph{Step 2: Passage to the limit (Minty formulation).}
Fix \((\mu,v,k)\in\cX^+\). By monotonicity of \(A_\epsilon\),
\[
\left\langle 
A_{\epsilon}\begin{bmatrix} \mu \\ v \\ k \end{bmatrix}
- A_{\epsilon}\begin{bmatrix} m_{\epsilon} \\ u_{\epsilon} \\ h_{\epsilon} \end{bmatrix},
\begin{bmatrix} \mu - m_{\epsilon} \\ v-u_{\epsilon} \\ k-h_{\epsilon}\end{bmatrix}\right\rangle
\ge 0.
\]
Combining this with \eqref{eq:varIneq_epLast} yields
\[
 \left\langle A_{\epsilon}\begin{bmatrix} \mu \\ v \\ k \end{bmatrix},
\begin{bmatrix} \mu - m_{\epsilon} \\ v-u_{\epsilon} \\ k-h_{\epsilon}\end{bmatrix}\right\rangle\geq 0.
\]
Expanding the pairing, 
letting \(\epsilon\to0^+\), and using weak convergence, 
we obtain
    \[
       \left\langle A_{0} \begin{bmatrix} \mu  \\ v   \\ k \end{bmatrix}, \begin{bmatrix} \mu - m   \\ v- u   \\ k-h  \end{bmatrix}  \right \rangle + \lim_{\epsilon \to 0^+} \epsilon^{\gamma'} \int_{\Gamma_D} k^{\gamma'-1} (k-h_{\epsilon})\ds\geq 0.
    \]
Here boundary terms involving $h_\epsilon$ are interpreted via the duality pairing 
$W^{-1+\frac1\gamma,\gamma'}(\Gamma_D)\times W^{1-\frac1\gamma,\gamma}(\Gamma_D)$,
and we use $h_\epsilon\rightharpoonup h$ in $W^{-1+\frac1\gamma,\gamma'}(\Gamma_D)$.    
The last term
vanishes in the limit. Indeed, recalling that \(h_\epsilon = \epsilon^{-\gamma}(u_\epsilon)_+^{\gamma-1}\), we estimate the \(L^{\gamma'}\) norm of \(h_\epsilon\):
\[
\|h_\epsilon\|_{L^{\gamma'}(\Gamma_D)}^{\gamma'} = \int_{\Gamma_D} \epsilon^{-\gamma\gamma'} (u_\epsilon)_+^{(\gamma-1)\gamma'} \ds = \epsilon^{-\gamma\gamma'} \int_{\Gamma_D} (u_\epsilon)_+^{\gamma} \ds,
\]
where we used that \((\gamma-1)\gamma' = \gamma\). Using the uniform bound from Proposition~\ref{pro:upBoundjuPSIneq2}, we have \(\int_{\Gamma_D} (u_\epsilon)_+^{\gamma} \ds \leq C \epsilon^\gamma\). Thus,
\[
\|h_\epsilon\|_{L^{\gamma'}(\Gamma_D)} \leq \left( C \epsilon^{-\gamma\gamma' + \gamma} \right)^{1/\gamma'} = C^{1/\gamma'} \epsilon^{-\gamma + \gamma/\gamma'} = C' \epsilon^{-1}.
\]
We can now bound the penalty term using H\"older's inequality. Since \(k \in L^{\gamma'}(\Gamma_D)\), we have \(k^{\gamma'-1} \in L^\gamma(\Gamma_D)\), and
\[
\left| \epsilon^{\gamma'} \int_{\Gamma_D} k^{\gamma'-1} h_\epsilon \ds \right| \leq \epsilon^{\gamma'} \|k^{\gamma'-1}\|_{L^\gamma(\Gamma_D)} \|h_\epsilon\|_{L^{\gamma'}(\Gamma_D)} \leq C'' \epsilon^{\gamma'} \epsilon^{-1} = C'' \epsilon^{\gamma'-1}.
\]
Since \(\gamma > 1\) implies \(\gamma' > 1\), the exponent \(\gamma'-1\) is positive, and the term converges to \(0\) as \(\epsilon \to 0^+\). Since \(\epsilon^{\gamma'} \int_{\Gamma_D} k^{\gamma'} \ds\) clearly vanishes, the entire limit term is zero.

Hence,
\[
\left\langle A_{0}\begin{bmatrix} \mu \\ v \\ k \end{bmatrix},
\begin{bmatrix} \mu - m \\ v-u \\ k-h\end{bmatrix}\right\rangle\geq 0
\qquad \forall (\mu,v,k)\in\cX^+,
\]
where we identify \(k\in L^{\gamma'}(\Gamma_D)\) with an element of \(W^{-1+\frac1\gamma,\gamma'}(\Gamma_D)\) via the natural embedding, so that \(k-h\in W^{-1+\frac1\gamma,\gamma'}(\Gamma_D)\) is well-defined and paired against traces in \(W^{1-\frac1\gamma,\gamma}(\Gamma_D)\).

\paragraph{Step 3: Conclusion.}
Since $A_0$ is monotone and hemicontinuous (by the same argument as for $A_\epsilon$) on the convex cone of admissible triples (allowing $h\in W^{-1+\frac1\gamma,\gamma'}(\Gamma_D)$), 
the Minty characterization of variational inequality solutions implies that the limit triplet $(m,u,h)$ satisfies
\[
\langle A_0(m,u,h), (\mu-m, v-u, k-h)\rangle\ge0\qquad\forall(\mu,v,k)\in\cX^+,
\]
which is exactly the statement of Theorem~\ref{thm:ExistInSepMFG}.
\end{proof}

\section{Conclusion} \label{sec:Conclusion}

We proved the existence of weak solutions for separable first-order MFGs with mixed boundary conditions via a monotone operator formulation.
To this end, we introduced an auxiliary boundary variable $h$ to encode the exit flux, enabling the formulation of the MFG system as a variational inequality on a convex domain. Then, we identified the failure of coercivity in the $u$-variable and resolved it through a quotient-space formulation that eliminates the constant-shift degeneracy. Finally, we established the existence of weak solutions by applying the Browder--Minty theorem to a family of penalized operators and passing to the limit.

The separable structure $\mathcal{H}(x,p,m) = H(x,p) - g(m)$ was essential to our approach: the strict monotonicity of $g$ allowed us to define the density explicitly as $m = g^{-1}(H(x,Du))$, which in turn enabled the reduction to the quotient-space operator $\bar{B}$.

For non-separable Hamiltonians $\mathcal{H}(x,p,m)$, one may still define the density implicitly via the relation $\mathcal{H}(x,Du,m) = 0$, yielding $m = m(x,Du)$ under appropriate monotonicity conditions on $\mathcal{H}$ with respect to $m$. However, this implicit formulation requires a substantially different structure for the assumptions---in particular, joint conditions on the growth and monotonicity of $\mathcal{H}$ in both $p$ and $m$---and leads to different analytical challenges in establishing the requisite bounds. These are addressed in a forthcoming companion paper.

\bibliographystyle{abbrv}
\bibliography{mfgv8_nn.bib}

\end{document}